\documentclass[10pt,reqno]{amsart}

\usepackage{hyperref}
\usepackage{amssymb, times}
\usepackage{amsmath}
\usepackage{amsfonts}
\usepackage{graphicx,cite}
\usepackage{latexsym}
\usepackage{enumerate, mathdots}
\usepackage{mathrsfs}
\usepackage{mathtools, subcaption}

\DeclarePairedDelimiter{\ceil}{\lceil}{\rceil}

\DeclareMathOperator{\stab}{Stab}
\renewcommand{\hat}{\widehat}

\newtheorem{Theorem}{Theorem}

\newtheorem{Lemma}{Lemma}
\theoremstyle{definition}
\newtheorem*{Definition}{Definition}
\newtheorem*{Example}{Example}

\numberwithin{equation}{section}

\newcommand{\supp}{\operatorname{supp}}

\newcommand{\inner}[1]{\left< #1 \right>}
\renewcommand{\pmod}[1]{\,(\operatorname{mod} #1)}
\newcommand{\C}{\mathbb{C}}
\newcommand{\R}{\mathbb{R}}

\newcommand{\Z}{\mathbb{Z}}
\newcommand{\X}{\mathcal{X}}
\newcommand{\Y}{\mathcal{Y}}
\newcommand{\Irr}{\operatorname{Irr}}

\renewcommand{\vec}[1]{{\boldsymbol{\bf #1}}}
\newcommand{\norm}[1]{\|#1\|}
\newcommand{\minimatrix}[4]{\begin{bmatrix}#1&#2\\#3&#4\end{bmatrix}}
\newcommand{\megamatrix}[9]{\begin{bmatrix} #1 & #2 & #3\\ #4 & #5 & #6 \\ #7 & #8 & #9 \end{bmatrix}}

\allowdisplaybreaks

\title{Supercharacters, exponential sums, and the uncertainty principle} 

\author[\tiny Brumbaugh]{J. L. Brumbaugh} 

\author[Bulkow]{Madeleine Bulkow} 
	\address{Department of Mathematics, Scripps College, 1030 Columbia Ave., Claremont, CA 91711} 

\author[Fleming]{Patrick S. Fleming}	
	\address{Mathematics and Computer Science Department, South Dakota School of Mines and Technology,
	501 East Saint Joseph Street  Rapid City, South Dakota 57701}
	\email{Patrick.Fleming@sdsmt.edu}

\author[Garcia]{Luis Alberto Garcia German} 
	\address{Department of Mathematics, Pomona College, 610 N. College Ave., Claremont, CA 91711} 
	\curraddr{Department of Mathematics, Washington University in St.~Louis, One Brookings Drive, St.~Louis, MO 63130-4899}
	\email{garciagerman@wustl.edu}

\author[Garcia]{Stephan Ramon Garcia}
	\address{Department of Mathematics, Pomona College, 610 N. College Ave., Claremont, CA 91711} 
	\email{Stephan.Garcia@pomona.edu}
	\urladdr{\url{http://pages.pomona.edu/~sg064747}}

\author[Karaali]{Gizem Karaali}
	\address{Department of Mathematics, Pomona College, 610 N. College Ave., Claremont, CA 91711} 
	\email{Gizem.Karaali@pomona.edu}
	\urladdr{\url{http://pages.pomona.edu/~gk014747}}

\author[Michal]{Matt Michal} \address{School of Mathematical Sciences, Claremont Graduate University, 150 E. 10th St., Claremont, CA 91711} 

\author[Suh]{Hong Suh}
	\address{Department of Mathematics, Pomona College, 610 N. College Ave., Claremont, CA 91711} 

\author[Turner]{Andrew P. Turner} \address{Department of Mathematics, Harvey Mudd College, 301 Platt Blvd., Claremont, CA 91711}
\curraddr{Center for Theoretical Physics, Massachusetts Institute of Technology, 77 Massachusetts Ave, 6-304, Cambridge, MA 02139}
\email{apturner@mit.edu}

\thanks{S.R.~Garcia was partially supported by NSF Grants DMS-1001614 and DMS-1265973.  
G.~Karaali was partially supported by a NSA Young Investigator
Award (NSA Grant \#H98230-11-1-0186).
P.S.~Fleming, S.R.~Garcia, and G.~Karaali were partially supported by the American Institute of Mathematics (AIM)
and NSF Grant DMS-0901523 via the REUF Program.  We also gratefully acknowledge the support of the
Fletcher Jones Foundation and Pomona College's SURP Program.}

\begin{document}

\begin{abstract}
The theory of supercharacters, which generalizes classical character theory, was recently introduced by P.~Diaconis and I.M.~Isaacs, building upon earlier work of C.~Andr\'e.  We study supercharacter theories on $(\Z/n\Z)^d$ induced by the actions of certain matrix groups, demonstrating that a variety of exponential sums of interest in number theory (e.g., Gauss, Ramanujan, Heilbronn, and Kloosterman sums) arise in this manner.  We develop a generalization of the discrete Fourier transform, in which supercharacters play the role of the Fourier exponential basis.  We provide a corresponding uncertainty principle and compute the associated constants in several cases.  
\end{abstract}

\keywords{Supercharacter, conjugacy class, superclass, circulant matrix, discrete Fourier transform, DFT, discrete cosine transform, DCT, Fourier transform,
Gauss sum, Gaussian period, Ramanujan sum, Heilbronn sum, Kloosterman sum, symmetric group, uncertainty principle}
\maketitle


\section{Introduction}

	The theory of \emph{supercharacters}, of which classical character theory is a special case, was recently introduced 
	by P.~Diaconis and I.M.~Isaacs in 2008 \cite{DiIs08}, generalizing the \emph{basic characters} studied by
	C.~Andr\'e \cite{An95, An01, An02}.  
	We are interested here in supercharacter theories on the group $(\Z/n\Z)^d$ induced by
	the action of certain subgroups $\Gamma$ of the group $GL_d(\Z/n\Z)$ of invertible
	$d \times d$ matrices over $\Z/n\Z$.  In particular, we demonstrate that a variety of
	exponential sums which are of interest in number theory arise as
	supercharacter values.  Among the examples we discuss are Gauss, Ramanujan, Heilbronn, 
	and Kloosterman sums.  Moreover, we also introduce a class of exponential sums
	induced by the natural action of the symmetric group $S_d$ on $(\Z/n\Z)^d$ that yields some visually striking patterns.
	
	In addition to showing that the machinery of supercharacter theory can be used to
	generate identities for certain exponential sums, we also develop a generalization
	of the discrete Fourier transform in which supercharacters play the
	role of the Fourier exponential basis.  For the resulting \emph{super-Fourier transform},
	we derive a supercharacter analogue of the uncertainty principle.	
	We also describe the algebra of all operators that are diagonalized by our transform.
	Some of this is reminiscent of 
	the theory of Fourier transforms of characteristic functions of orbits in Lie algebras
	over finite fields \cite[Lem.~3.1.10]{Letellier}, \cite[Lem.~4.2]{Lehrer}, \cite{Springer}.
	
	Although it is possible to derive some of our results by considering the classical character theory of the semidirect
	product $(\Z/n\Z)^d \rtimes \Gamma$, the supercharacter approach is cleaner and more natural.  The character
	tables produced via the classical approach are typically large and unwieldy, containing many entries that are irrelevant
	to the study of the particular exponential sum being considered.
	This is a reflection of the fact that $(\Z/n\Z)^d \rtimes \Gamma$ is generally nonabelian and possesses a large
	number of conjugacy classes.
	On the other hand, our supercharacter tables
	are smaller and simpler than their classical counterparts.  
	Indeed, the supercharacter approach takes place entirely inside the original abelian group $(\Z/n\Z)^d$,
	that possesses only a few superclasses.  
	
	We cover the preliminary definitions and notation in Section \ref{SectionPreliminaries},
	before introducing the super-Fourier transform in Section \ref{SectionFourier}.  A number of examples, including
	those involving Gauss, Kloosterman, Heilbronn, and Ramanujan sums, are discussed in Section \ref{SectionExponential}.
	We conclude this note with a few words concerning an extension of our technique to more general matrix groups
	in Section \ref{SectionJ}.
	
\section{Supercharacter theories on $(\Z/n\Z)^d$}\label{SectionPreliminaries}

	To get started, we require the following important definition.

	\begin{Definition}[Diaconis-Isaacs \cite{DiIs08}]
		Let $G$ be a finite group, let $\X$ be a partition of the set $\Irr G$ of irreducible characters of $G$, 
		and let $\Y$ be a partition of $G$.  We call the ordered pair $(\X, \Y)$ a \emph{supercharacter theory} if 
		\begin{enumerate}\addtolength{\itemsep}{0.5\baselineskip}
			\item[(i)] $\Y$ contains $\{1\}$, where $1$ denotes the identity element of $G$,
			\item[(ii)] $|\X| = |\Y|$,
			\item[(iii)] For each $X$ in $\X$, the character  
				\begin{equation}\label{eq-SCD}
					\sigma_X = \sum_{\chi \in X} \chi(1)\chi
				\end{equation}
				is constant on each $Y$ in $\Y$.
		\end{enumerate}
		The characters $\sigma_X$ are called \emph{supercharacters} and the elements $Y$ of $\Y$
		 are called \emph{superclasses}.
	\end{Definition}
	
	If $(\X,\Y)$ is a supercharacter theory on $G$, then it turns out that each $Y$
	in $\Y$ must be a union of conjugacy classes.  One can also show that the partitions
	$\X$ and $\Y$ uniquely determine each other and, moreover, that the set $\{ \sigma_X : X \in \X\}$
	forms a basis for the space $\mathcal{S}$ of \emph{superclass functions} on $G$ (i.e., functions $f:G\to\C$ which are
	constant on each superclass).
	
	Let us now say a few words about our notation.
	We let $\vec{x} = (x_1,x_2,\ldots,x_d)$ and $\vec{y}=(y_1,y_2,\ldots,y_d)$ 
	denote elements of $G:=(\Z/n\Z)^d$ and we write $\vec{x} \cdot \vec{y} := \sum_{i=1}^d x_i y_i$ so that
	$A\vec{x} \cdot \vec{y} = \vec{x} \cdot A^T \vec{y}$ for all $\vec{x}, \vec{y}$ in $G$ and
	each $A$ in $GL_d(\Z/n\Z)$. 
	The symbol $\vec{\xi}$ will frequently be used to distinguish a vector 
	that is to be regarded as the argument of a function on $G$.
	Since we will ultimately deal with a variety of exponential sums, we also
	let $e(x) := \exp(2 \pi i x)$, so that the function $e(x)$ is periodic with period $1$.

	In the following, $\Gamma$ denotes a symmetric (i.e., $\Gamma^T = \Gamma$) subgroup of $GL_d(\Z/n\Z)$.
	For each such $\Gamma$, we construct a corresponding supercharacter theory on
	$G$ using the following recipe.
	The superclasses $Y$ are simply the orbits $\Gamma \vec{y}$ in $G$ under the action $\vec{y}^A := A\vec{y}$ of $\Gamma$.  
	Among other things, we note that $\{\vec{0}\}$ is an orbit of this action so that axiom (i) in the Diaconis-Isaacs definition is satisfied.
	The corresponding supercharacters require a bit more work to describe. 
	
	We first recall that $\Irr G = \{ \psi_{\vec{x}} : \vec{x} \in G \}$, in which
	\begin{equation}\label{eq-Psi}
		\psi_{\vec{x}}(\vec{\xi}) = e \left( \frac{\vec{x} \cdot \vec{\xi}}{n} \right).
	\end{equation}
	We now let $\Gamma$ act upon $\Irr G$ by setting
	\begin{equation}\label{eq-InverseTranspose}
		\psi_{\vec{x}}^A := \psi_{A^{-T}\vec{x}},
	\end{equation}
	where $A^{-T}$ denotes the inverse transpose of $A$.  In light of the fact that
	\begin{equation*}
		\psi_{\vec{x}}^{AB} = \psi_{(AB)^{-T} \vec{x}}= \psi_{A^{-T} B^{-T} \vec{x}}
		=  \psi^A_{B^{-T} \vec{x}} = (\psi^B_{\vec{x}})^A,
	\end{equation*}
	it follows that \eqref{eq-InverseTranspose} defines a group action.  Since
	\begin{equation*}
		\psi^A_{\vec{x}}( \vec{y}^A) = 
		e \left( \frac{ A^{-T}\vec{x} \cdot A\vec{y} }{n} \right) = 
		e \left( \frac{\vec{x} \cdot \vec{y}}{n} \right) =
		\psi_{\vec{x}}(\vec{y}),
	\end{equation*}
	it follows from a result of Brauer \cite[Thm.~6.32, Cor.~6.33]{Isaacs} that the
	actions of $\Gamma$ on $G$ and on $\Irr G$ yield the same number of orbits.
	Letting $\X$ denote the set of orbits in $\Irr G$ and $\Y$ denote the set of orbits
	in $G$, we set
	\begin{equation*}
		N:=|\X| = |\Y|.
	\end{equation*}
	In particular, condition (ii) holds.
		
	Although the elements of each orbit $X$ in $\X$ are certain characters $\psi_{\vec{x}}$,
	we shall agree to identify $\psi_{\vec{x}}$ with the corresponding vector $\vec{x}$ so that the set $X$
	is stable under the action $\vec{x} \mapsto A^{-T} \vec{x}$ of $\Gamma$.
	Having established this convention, 
	for each $X$ in $\X$ we follow \eqref{eq-SCD} and define the corresponding character
	\begin{equation}\label{eq-Supercharacter}
		\sigma_X(\vec{\xi}) = \sum_{\vec{x} \in X} e \left( \frac{ \vec{x} \cdot \vec{\xi} } {n} \right).
	\end{equation}
	We claim that the characters $\sigma_X$ are constant on each superclass $\Gamma \vec{y}$.
	Indeed, if $\vec{y}_1 = A\vec{y}_2$ for some $A$ in $\Gamma$, then 
	\begin{equation*}
		\sigma_X(\vec{y}_1) 
		= \sum_{\vec{x} \in X} e \left( \frac{ \vec{x} \cdot \vec{y}_1 } {n} \right) 
		= \sum_{\vec{x} \in X} e \left( \frac{ A^T\vec{x} \cdot \vec{y}_2 } {n} \right) 
		= \sum_{\vec{x}' \in X} e \left( \frac{ \vec{x}' \cdot \vec{y}_2 } {n} \right) 
		= \sigma_X(\vec{y}_2).
	\end{equation*}
	Therefore condition (iii) holds.  Putting this all together, we conclude that the pair $(\X,\Y)$ 
	constructed above is a supercharacter theory on $G$.
	
	We henceforth refer to the characters $\sigma_X$ as \emph{supercharacters} and the sets $Y$
	as \emph{superclasses}.  Expanding upon the notational conventions introduced above,
	we choose to identify the set $X$, whose elements are the irreducible characters that comprise
	$\sigma_X$, with the set of vectors $\{ \vec{x} : \psi_{\vec{x}} \in X\}$.  
	Having made this identification, we see
	that $\X = \Y$ since the condition $\Gamma = \Gamma^T$ ensures that
	the orbits in $G$ under the actions $\vec{x} \mapsto A\vec{x}$ and $\vec{x} \mapsto A^{-T}\vec{x}$ coincide.
	In light of this, we shall frequently regard the elements $X$ of $\X$ as superclasses.
	
	Since $\sigma_X$ is constant on each superclass $Y$, if $\vec{y}$ belongs to $Y$ we will often write
	$\sigma_X(Y)$ instead of $\sigma_X(\vec{y})$.
	Let us also note that the negative $-X := \{ -\vec{x} : \vec{x} \in X\}$ of a superclass $X$
	is also a superclass.   In particular,
	\begin{equation}\label{eq-ConjugatePairs}
		\sigma_{-X}(Y) = \overline{\sigma_X(Y)},
	\end{equation}	
	so that the complex conjugate of a supercharacter is itself a supercharacter.
	Another fact which we shall make use of is the obvious inequality
	\begin{equation}\label{eq-Max}
		|\sigma_X(\vec{\xi})| \leq |X|.
	\end{equation}

	In addition to \eqref{eq-Supercharacter}, there is another description of the supercharacters $\sigma_X$ that
	is more convenient in certain circumstances.  Letting
	\begin{equation*}
		\stab(\vec{x}) := \{ A \in \Gamma : A \vec{x} = \vec{x} \},
	\end{equation*}
	it follows that the orbit $X = \Gamma \vec{x}$ contains $|\stab(\vec{x})|$ copies of $\vec{x}$ so that
	\begin{equation} \label{eq-Stab}
		\sigma_X(\vec{\xi}) = \frac{1}{|\stab(\vec{x})|}\sum_{A \in \Gamma} e \left( \frac{  A  \vec{x} \cdot \vec{\xi} } {n} \right),
	\end{equation}
	since $\Gamma$ is closed under inversion.
	
	We now fix an enumeration $X_1,X_2,\ldots,X_N$ of $\X = \Y$ and label the 
	supercharacters corresponding to these sets $\sigma_1,\sigma_2,\ldots,\sigma_N$.	
	Recall that $L^2(G)$, the space of complex-valued functions on $G=(\Z/n\Z)^d$,
	is endowed with the inner product
	\begin{equation}\label{eq-StandardInnerProduct}
		\inner{f,g} =  \sum_{\vec{x} \in G} f(\vec{x})\overline{g(\vec{x})},
	\end{equation}
	with respect to which the irreducible characters \eqref{eq-Psi} form an orthogonal set.  We then have
	\begin{equation}\label{eq-ssmn}
		\inner{ \sigma_i, \sigma_j} = n^d |X_i| \delta_{i,j}.
	\end{equation}
	On the other hand, since supercharacters are constant on superclasses, we also have
	\begin{equation}\label{eq-SuperInnerProduct}
		\inner{\sigma_i,\sigma_j} 
		=  \sum_{\ell=1}^N |X_{\ell}| \sigma_i(X_{\ell})\overline{\sigma_j(X_{\ell})}.
	\end{equation}
	Comparing \eqref{eq-ssmn} and \eqref{eq-SuperInnerProduct}, we conclude that the $N \times N$ matrix	
	\begin{equation}\label{eq-U}
		U = \frac{1}{\sqrt{n^d}} \left[  \frac{   \sigma_i(X_j) \sqrt{  |X_j| }}{ \sqrt{|X_i|}} \right]_{i,j=1}^N
	\end{equation}
	is unitary.  The properties of this matrix are summarized in the following lemma.

	\begin{Lemma}\label{LemmaU}
		The unitary matrix $U$ given by \eqref{eq-U} satisfies the following.
		\begin{enumerate}\addtolength{\itemsep}{0.5\baselineskip}
			\item $U = U^T$, or equivalently
				\begin{equation}\label{eq-ReciprocityFormula}
					\frac{\sigma_i(X_j)}{|X_i|} = \frac{\sigma_j(X_i)}{|X_j|} ,
				\end{equation}
			\item $U^2 = P$, the permutation matrix that interchanges positions $i$ and $j$ whenever
				$X_i = - X_j$ and fixes position $i$ whenever $X_i = -X_i$,
			\item $U^4 = I$.
		\end{enumerate}
	\end{Lemma}

	\begin{proof}
		Letting $X_i = \Gamma \vec{x}_i$ and $X_j = \Gamma \vec{x}_j$, we use
		\eqref{eq-Stab} to find that
		\begin{align*}
			|\stab{\vec{x}}_i|\sigma_{i}(X_j)
			= \sum_{A \in \Gamma} e \left( \frac{A\vec{x}_i \cdot \vec{x}_j}{n} \right) 
			= \sum_{A^T \in \Gamma} e \left( \frac{A^T\vec{x}_j \cdot \vec{x}_i }{n} \right)  
			= |\stab{\vec{x}}_j|\sigma_{j}(X_i).
		\end{align*}
		We conclude from the orbit-stabilizer theorem that
		\begin{equation*}
			\frac{|\Gamma|}{|X_i|}\sigma_{i}(X_j) = \frac{|\Gamma|}{|X_j|}\sigma_{j}(X_i),
		\end{equation*}
		which implies that $U = U^T$.  In light of \eqref{eq-ConjugatePairs}, it follows that
		$\overline{U} = PU$.  Noting that
		$P = P^{-1}$, we find that $I = U^*U = \overline{U}U = PU^2$,
		so that $U^2 = P$ and $U^4 = I$.
	\end{proof}


\section{The super-Fourier transform}\label{SectionFourier}
	In this section we develop supercharacter generalizations of the discrete Fourier transform (DFT).
	Maintaining the notation and conventions established in the preceding section, we let $\X = \Y = \{X_1,X_2,\ldots,X_N\}$
	and let $\sigma_1,\sigma_2,\ldots,\sigma_N$ denote the corresponding supercharacters.		
	Let $\mathcal{S} \subset L^2(G)$ denote set of all superclass functions, equipped with the inherited $L^2(G)$ norm
	\begin{equation*}
		\norm{f} := \Big(\sum_{\ell=1}^N |X_{\ell}| |f(X_{\ell})|^2 \Big)^{\frac{1}{2}}.
	\end{equation*}
	We will often regard a function $f \in \mathcal{S}$ as a function $f:\X\to \C$.
	
	By analogy with the discrete Fourier transform, we would like to find a superclass function $\widehat{f}$
	that satisfies the \emph{inversion formula}
	\begin{equation}\label{eq-Inversion}
		f = \frac{1}{\sqrt{n^d}} \sum_{\ell=1}^N \widehat{f}(X_{\ell}) \sigma_{\ell},
	\end{equation}
	the normalization factor being included to ensure the unitarity of the map $f \mapsto \widehat{f}$ (see Theorem \ref{TheoremF} below).
	In light of \eqref{eq-ssmn} and the reciprocity formula \eqref{eq-ReciprocityFormula}, it follows that
	\begin{equation*}
		\widehat{f}(X_i) 
		= \sqrt{n^d}\frac{\inner{f, \sigma_i}}{\inner{\sigma_i,\sigma_i}} 
		= \sum_{\ell=1}^N   \frac{|X_{\ell}| f(X_{\ell}) \overline{\sigma_i(X_{\ell})}}{\sqrt{n^d}|X_i|} 
		= \frac{1}{\sqrt{n^d}} \sum_{\ell=1}^N  f(X_{\ell}) \overline{\sigma_{\ell}(X_i)} .
	\end{equation*}
	We therefore define the \emph{super-Fourier transform} of the superclass function $f$
	(induced by the action of $\Gamma$ on $(\Z/n\Z)^d$) by setting
	\begin{equation}\label{eq-Fourier}
		\widehat{f} := \frac{1}{\sqrt{n^d}} \sum_{\ell=1}^N  f(X_{\ell}) \overline{\sigma_{\ell}} .
	\end{equation}
	The linear operator $\mathscr{F}: \mathcal{S} \to \mathcal{S}$ defined by $\mathscr{F}f = \widehat{f}$
	will also be referred to as the \emph{super-Fourier transform}.
	
	Although the formulas \eqref{eq-Inversion} and \eqref{eq-Fourier} resemble familiar formulas
	involving the discrete Fourier transform, we have not yet justified that this resemblance is more than superficial.
	We next show that the super-Fourier transform indeed enjoys several of the standard algebraic properties
	of the DFT.
	
	Normalizing each of the supercharacters $\sigma_i$, we obtain the orthonormal basis $\{s_1,s_2,\ldots,s_N\}$
	of $\mathcal{S}$ whose elements are defined by
	\begin{equation}\label{eq-s}
		s_i = \frac{\sigma_i}{\sqrt{n^d |X_i|}}.
	\end{equation}
	With respect to this basis we have the expansions	
	\begin{equation}\label{eq-FTS}
		f = \sum_{\ell=1}^N \sqrt{|X_{\ell}|} \widehat{f}(X_{\ell}) s_{\ell},\qquad
		\widehat{f} = \sum_{\ell=1}^N \sqrt{|X_{\ell}|} f(X_{\ell}) \overline{s_{\ell}}.
	\end{equation}
	Computing the $(i,j)$ entry in the matrix representation of $\mathscr{F}$ with respect to the basis $\{s_1,s_2,\ldots,s_N\}$ 
	shows that	
	\begin{align*}
		\inner{ \mathscr{F}s_j,s_{i}}
		&= \left< \sum_{\ell=1}^N \sqrt{|X_{\ell}|} s_j(X_{\ell}) \overline{s_{\ell}}, s_i\right> && \text{by \eqref{eq-FTS}}\\
		&=  \sum_{\ell=1}^N \sqrt{|X_{\ell}|} s_j(X_{\ell}) \left<  \overline{s_{\ell}}, s_i\right> \\
		&= \sqrt{|-X_i|} s_j(-X_i) && \text{by \eqref{eq-ConjugatePairs}}\\
		&= \frac{ \overline{\sigma_j(X_i)} \sqrt{ |X_i|} }{\sqrt{n^d |X_j|} }. && \text{by \eqref{eq-s}}
	\end{align*}
	In other words, the matrix representation for $\mathscr{F}$
	with respect to the orthonormal basis \eqref{eq-s}
	is precisely the unitary matrix $U^*$.  At this point, most of the following theorem is a direct
	consequence of Lemma \ref{LemmaU}.

	\begin{Theorem}\label{TheoremF}
		Let $\Gamma = \Gamma^T$ be a subgroup of $GL_d(\Z/n\Z)$ and let $\X = \{X_1,X_2,\ldots,X_N\}$
		denote the set of superclasses induced by the action of $\Gamma$ on $(\Z/n\Z)^d$.  The
		super-Fourier transform satisfies the following:
		\begin{enumerate}\addtolength{\itemsep}{0.5\baselineskip}
			\item $\norm{\widehat{f}} = \norm{f}$,
			\item $[\mathscr{F}^2f](X) = f(-X)$ for every $X$ in $\X$,
			\item $\mathscr{F}^4f = f$.
		\end{enumerate}
		Moreover, if $f \in \mathcal{S}$ is not identically zero, then
		\begin{equation}\label{eq-Uncertainty}
			\ceil[\Big]{\frac{n^d}{M} } \leq |\supp f| |\supp \widehat{f} | ,
		\end{equation}
		where $\ceil{\,\cdot\,}$ denotes the ceiling function,		
		$M = \displaystyle \max_{1\leq i\leq N} | X_i|$, and
		\begin{equation*}
			\supp f = \{ X \in \X : f(X) \neq 0\}.
		\end{equation*}
	\end{Theorem}

	\begin{proof}
		It suffices to prove \eqref{eq-Uncertainty} (note that $|\supp f|$ denotes the number of superclasses $X \in \X$
		for which $f(X) \neq 0$).
		For $f \in \mathcal{S}$, let $\norm{f}_{\infty} = \max_{1\leq i \leq N} |f(X_i)|$.
		Using \eqref{eq-Max}, we find that
		\begin{align*}
			\norm{ \widehat{f} }_{\infty}
			&= \max_{1\leq i \leq N} | \widehat{f}(X_i) | \\
			&= \max_{1\leq i \leq N} \left| \frac{1}{\sqrt{n^d}} \sum_{\ell=1}^N f(X_{\ell}) \overline{\sigma_{\ell}(X_i)} \right|  \\
			&\leq \frac{1}{\sqrt{n^d}} \max_{1\leq i \leq N}  \sum_{\ell=1}^N |f(X_{\ell})|  |\sigma_{\ell}(X_i) |  \\
			&\leq \sqrt{\frac{M}{n^d}} \sum_{\ell=1}^N |X_{\ell} |^{\frac{1}{2}} |f(X_{\ell})|    \\
			&\leq \sqrt{\frac{M}{n^d}} |\supp f|^{\frac{1}{2}}
				\left(\sum_{\ell=1}^N |X_{\ell}| |f(X_{\ell})|^2 \right)^{\frac{1}{2}}\\
			&= \sqrt{\frac{M}{n^d}} |\supp f|^{\frac{1}{2}} \norm{f} \\
			&= \sqrt{\frac{M}{n^d}} |\supp f|^{\frac{1}{2}} \norm{\widehat{f}} \\
			&\leq \sqrt{\frac{M}{n^d}} |\supp f|^{\frac{1}{2}} |\supp \widehat{f}|^{\frac{1}{2}}\norm{\widehat{f}}_{\infty} ,
		\end{align*}
		which implies the desired result since $|\supp f|$ and $|\supp \widehat{f}|$ are positive integers.
	\end{proof}
	
	Recall that the classical Fourier-Plancherel transform $f\mapsto \widehat{f}$ on $L^2(\R)$ satisfies 
	the important identity
	\begin{equation}\label{eq-FPT}
		\hat{f'\,}\!(\xi) = 2 \pi i \xi \hat{f}(\xi)
	\end{equation}
	on a dense subset of $L^2(\R)$.
	To be more specific, the Fourier-Plancherel transform provides us with the
	spectral resolution of the unbounded operator $f \mapsto f'$.
	This observation is crucial, for instance, in the study of partial differential equations and in the
	development of pseudo-differential operators.  
	
	We now consider analogues of the identity \eqref{eq-FPT} for the super-Fourier transform $\mathscr{F}:\mathcal{S} \to \mathcal{S}$.  
	Recalling that the unitary matrix $U^*$, defined by \eqref{eq-U}, is the matrix representation 
	of $\mathscr{F}$ with respect to the orthonormal basis $\{s_1,s_2,\ldots,s_N\}$ of $\mathcal{S}$,
	we identify operators on $\mathcal{S}$ with their matrix representations with respect to this basis.
	We therefore seek to classify all $N \times N$ matrices $T$ that satisfy
	\begin{equation}\label{eq-FTDF}
		 T U = U D
	\end{equation}
	for some diagonal matrix $D$.  A complete characterization of such matrices is provided by our next
	theorem, which is inspired by a result from classical character theory 
	\cite[Section 33]{CR62}, \cite[Lem.~3.1]{CKS},  \cite[Lem.~4]{Ku75}.
	Portions of the following proof originate in \cite{RSS}, where a notion of superclass arithmetic is
	developed for arbitrary finite groups.  However, in that more general context the corresponding conclusions
	are not as strong as those given below.

	\begin{Theorem}\label{TheoremPowerful}
		Let $\Gamma = \Gamma^T$ be a subgroup of $GL_d(\Z/n\Z)$, let $\X = \{X_1,X_2,\ldots,X_N\}$
		denote the set of superclasses induced by the action of $\Gamma$ on $(\Z/n\Z)^d$, and let $\sigma_1,\sigma_2,\ldots,\sigma_N$
		denote the corresponding supercharacters.
		For each fixed $z$ in $X_k$, let $c_{i,j,k}$ denote the number of solutions $(x_i,y_j) \in X_i \times X_j$ to the equation $x+y = z$.
		\begin{enumerate}\addtolength{\itemsep}{0.5\baselineskip}
			\item $c_{i,j,k}$ is independent of the representative $z$ in $X_k$ which is chosen.
			\item The identity
				\begin{equation}\label{eq-sijk}
					\sigma_i(X_{\ell}) \sigma_j(X_{\ell}) = \sum_{k=1}^N c_{i,j,k} \sigma_k(X_{\ell})
				\end{equation}			
				holds for $1\leq i,j,k,\ell \leq N$.
			\item The matrices $T_1,T_2,\ldots,T_N$, whose entries are given by
				\begin{equation*}
					[T_i]_{j,k} = \frac{ c_{i,j,k} \sqrt{ |X_k| } }{ \sqrt{ |X_j|} },
				\end{equation*}
				each satisfy 
				\begin{equation}\label{eq-TUUD}
					T_i U = U D_i,
				\end{equation}
				where 
				\begin{equation}\label{eq-DM}
					D_i = \operatorname{diag}\big(\sigma_i(X_1), \sigma_i(X_2),\ldots, \sigma_i(X_N) \big).
				\end{equation}
				
			\item Each $T_i$ is a normal matrix (i.e., $T_i^*T_i = T_i T_i^*$) and
				the set $\{T_1,T_2,\ldots,T_N\}$ forms a basis for the algebra $\mathcal{A}$ of all $N \times N$ matrices 
				$T$ such that $U^*TU$ is diagonal.
		\end{enumerate}
  	\end{Theorem}
	
	\begin{proof}
		The fact that the structure constants $c_{i,j,k}$ do not depend upon the representative $z$ of $X_k$ is mentioned in
		passing in \cite[Cor.~2.3]{DiIs08}; a complete proof can be found in \cite{RSS}.  Let us now focus our attention on 
		\eqref{eq-sijk}.
		We work in the group algebra $\C[(\Z/n\Z)^d]$, noting that each character of $(\Z/n\Z)^d$
		extends by linearity to a function on the entire group algebra.  For each superclass $X_i$ we let
		\begin{equation*}
			\tilde{X}_i = \sum_{\vec{x}\in X_i} \vec{x}
		\end{equation*}
		denote the corresponding superclass sum in $\C[(\Z/n\Z)^d]$, remarking for emphasis
		that $\tilde{X}_i$ is to be regarded as a formal sum of the elements of $X_i$.
		It is easy to see that these superclass sums satisfy
		\begin{equation}\label{eq-ijk}
			\tilde{X}_i \tilde{X}_j = \sum_{k=1}^N c_{i,j,k} \tilde{X}_k.
		\end{equation}
		
		We now claim that for $1 \leq j \leq N$, the irreducible characters \eqref{eq-Psi} satisfy
		\begin{equation}\label{eq-SVS}
			\psi_{\vec{x}}( \tilde{X}_j ) = \psi_{\vec{x}'}( \tilde{X}_j ),
		\end{equation}
		whenever $\vec{x}$ and $\vec{x}'$ belong to the same superclass.
		Indeed, under this hypothesis there exists a matrix $A$ in $\Gamma$ such that $\psi_{\vec{x}} = \psi_{A^{-T}\vec{x}'}$, whence
		\begin{equation*}
			\psi_{\vec{x}}( \tilde{X}_j )
			= \!\!\sum_{\vec{v} \in X_j} \!\!\psi_{\vec{x}}(\vec{v}) 
			= \!\!\sum_{\vec{v} \in X_j} \!\!\psi_{A^{-T}\vec{x}'}(\vec{v}) 
			= \!\!\sum_{\vec{v} \in X_j} \!\!\psi_{\vec{x}'}(A^{-1}\vec{v}) 
			= \!\!\sum_{\vec{v}' \in X_j} \!\!\psi_{\vec{x}'}(\vec{v}') 
			= \psi_{\vec{x}'}( \tilde{X}_j)
		\end{equation*}
		since $X_j$ is stable under the action of $\Gamma$.  
		If $\vec{x}$ belongs to $X_{\ell}$, then \eqref{eq-SVS} implies that
		\begin{equation}\label{eq-ILO}
			|X_{\ell}| \psi_{\vec{x}}(\tilde{X}_j) 
			= \sum_{\vec{x}'\in X_{\ell}} \psi_{\vec{x}'}( \tilde{X}_j) 
			= \sigma_{\ell}( \tilde{X}_j) 
			= |X_j| \sigma_{\ell}(X_j)
		\end{equation}
		since $\sigma_{\ell}$ is constant on the superclass $X_j$.
		Applying $\psi_{\vec{x}}$ to \eqref{eq-ijk} we obtain
		\begin{equation*}
			\psi_{\vec{x}}(\tilde{X}_i) \psi_{\vec{x}}( \tilde{X}_j) = \sum_{k=1}^N c_{i,j,k} \psi_{\vec{x}}(\tilde{X}_k),
		\end{equation*}
		from which
		\begin{equation*}
			\frac{ |X_i| \sigma_{\ell}(X_i)}{|X_{\ell}|} \cdot\frac{ |X_j| \sigma_{\ell}(X_j)}{|X_{\ell}|}
			= \sum_{k=1}^N c_{i,j,k} \frac{ |X_k| \sigma_{\ell}(X_k)}{|X_{\ell}|}
		\end{equation*}
		follows by \eqref{eq-ILO}.  In light of the reciprocity formula \eqref{eq-ReciprocityFormula},
		we conclude that \eqref{eq-sijk} holds for $1\leq i,j,k,\ell\leq N$.
		
		In terms of matrices, we see that \eqref{eq-sijk} is simply the $(j,\ell)$ entry of the matrix equation
		$M_i W = W D_i$, in which $M_i = [c_{i,j,k}]_{j,k=1}^N$ and $W = [\sigma_j(X_k)]_{j,k=1}^N$.
		Conjugating all of the matrices involved by an appropriate diagonal matrix yields 
		\eqref{eq-TUUD}.

		Since we are dealing with $N \times N$ matrices, it is clear that the algebra $\mathcal{A}$ of all $N \times N$
		matrices $T$ such that $U^*TU$ is diagonal has dimension at most $N$.  Because 
		the $D_i$ are linearly independent (this follows from the fact that the rows of $W$ are linearly independent since $W$ is similar to 
		the unitary matrix $U$), 
		it follows that $\mathcal{A} = \operatorname{span}\{T_1,T_2,\ldots,T_N\}$.
	\end{proof}

\section{Exponential sums}\label{SectionExponential}
	In this section we examine a number of examples of the preceding machinery.	
	In particular, we focus on several classes of exponential sums that are relevant in
	number theory (e.g., Gauss, Ramanujan, Heilbronn, and Kloosterman sums).
	Although it is certainly possible to explore the specific properties of these sums using
	Theorems \ref{TheoremF} and \ref{TheoremPowerful} (see \cite{RSS, HES}),
	that is not our purpose here.
	We simply aim to demonstrate how such sums arise in a natural and unified
	manner from the theory of supercharacters.

\subsection{Maximum collapse}
		If $G = (\Z/n\Z)^d$ and $\Gamma = GL_d(\Z/n\Z)$, then $\X = \Y = \big\{ \{ \vec{0}\}, G \backslash \{ \vec{0}\} \big\}$.
		The corresponding supercharacter table and symmetric unitary matrix are displayed below.
		\begin{equation*}\small
			\begin{array}{|c|cc|}
				\hline
				(\Z/n\Z)^d & \{ \vec{0} \} & G \backslash \{ \vec{0}\}\\
				GL_d(\Z/n\Z)& \vec{0} & (1,1,\ldots,1)\\
				\# & 1 & n^d - 1\\
				\hline
				\sigma_1 & 1 & 1 \\
				\sigma_2 & n^d-1 & -1 \\
				\hline
			\end{array}
			\qquad\qquad\normalsize
			\underbrace{
			\frac{1}{\sqrt{n^d}}
			\begin{bmatrix}
				1 & \sqrt{n^d-1} \\
				\sqrt{n^d-1} & -1
			\end{bmatrix}
			}_U
		\end{equation*}
		In this setting the uncertainty principle \eqref{eq-Uncertainty} takes the form 
		$2 \leq |\supp f| |\supp \widehat{f}|$, which is obviously sharp.

\subsection{The discrete Fourier transform}
		If $G = \Z/n\Z$ and $\Gamma = \{1\}$, then $\X = \Y = \big\{ \{ x\}: x \in \Z/n\Z\big\}$. 
		The corresponding supercharacter table and associated unitary matrix are displayed below
		(where $\zeta = \exp(2\pi i /n)$).
		\begin{equation*}\footnotesize
			\begin{array}{|c|ccccc|}
				\hline
				\Z/n\Z & \{0\} & \{1\} & \{2\} & \cdots & \{n-1\} \\
				\{1\} & 0 & 1 & 2 & \cdots & n-1 \\
				\# & 1 &1 & 1 & \cdots & 1 \\
				\hline
				\sigma_0 & 1 & 1 & 1 & \cdots & 1 \\
				\sigma_1 & 1 & \zeta & \zeta^2 & \cdots & \zeta^{n-1} \\
				\sigma_2 & 1 & \zeta^2 & \zeta^4 & \cdots & \zeta^{2(n-1)} \\
				\vdots & \vdots & \vdots & \vdots & \ddots & \vdots \\
				\sigma_{n-1} & 1 & \zeta^{n-1} & \zeta^{2(n-1)} & \cdots & \zeta^{(n-1)^2} \\
				\hline
			\end{array}
			\qquad
			\underbrace{
			\frac{1}{\sqrt{n}}\!\!
			\begin{bmatrix}
				1 & 1 & 1 & \cdots & 1 \\
				1 & \zeta & \zeta^2 & \cdots & \zeta^{n-1} \\
				1 & \zeta^2 & \zeta^4 & \cdots & \zeta^{2(n-1)} \\
				\vdots & \vdots & \vdots & \ddots & \vdots \\
				1 & \zeta^{n-1} & \zeta^{2(n-1)} & \cdots & \zeta^{(n-1)^2} \\
			\end{bmatrix}
			}_U
		\end{equation*}
		In particular, $U$ is the \emph{discrete Fourier transform} (DFT) matrix.
		If we agree to identify each superclass $\{x\}$ with the corresponding element $x$ in $\Z/n\Z$, 
		then the super-Fourier transform is simply the discrete Fourier transform
		\begin{equation*}
			[\mathscr{F}f](\xi) = \frac{1}{\sqrt{n}} \sum_{j=1}^n f(j) e^{-2\pi i j \xi/n}
		\end{equation*}
		and \eqref{eq-Uncertainty} is the standard Fourier uncertainty principle:
		$n \leq |\supp f| |\supp \widehat{f}|$.
		More generally, if $G = (\Z/n\Z)^d$ and $\Gamma = \{I\}$, then every superclass is again a singleton 
		whence \eqref{eq-Uncertainty} yields the familiar estimate $|G| \leq |\supp f| |\supp \widehat{f}|$
		(see Subsection \ref{SubsectionTao} for a relevant discussion).
		
		Turning our attention toward Theorem \ref{TheoremPowerful}, we find that the matrices
		\begin{equation*}
			[T_i]_{j,k} = 
			\begin{cases}
			0 & \text{if $k - j \neq i$},\\
			1 & \text{if $k - j = i$},
			\end{cases}
		\end{equation*}
		each satisfy $T_i U= U D_i$ where
		$D_i = \operatorname{diag}(1, \zeta^{i}, \zeta^{2i},\ldots, \zeta^{(n-1)i})$.  Moreover,
		the algebra $\mathcal{A}$ generated by the $T_i$ is precisely the algebra of all
		$N \times N$ \emph{circulant matrices}
		\begin{equation*}
			\begin{bmatrix}
				c_0     & c_{N-1} & \cdots  & c_{2} & c_{1}  \\
				c_{1} & c_0    & c_{N-1} &         & c_{2}  \\
				\vdots  & c_{1}& c_0    & \ddots  & \vdots   \\
				c_{N-2}  &        & \ddots & \ddots  & c_{N-1}   \\
				c_{N-1}  & c_{N-2} & \cdots  & c_{1} & c_0 \\
			\end{bmatrix}.
		\end{equation*}

\subsection{The discrete cosine transform}
	If $G = \Z/n\Z$ and $\Gamma = \{ \pm 1\}$, then
	\begin{equation*}
		\X = 
		\begin{cases}
		\big\{ \{0\},\, \{\pm1\},\, \{ \pm 2\}, \ldots,\{\tfrac{n}{2}\pm 1\},\,  \{ \tfrac{n}{2} \} \big\} & \text{if $n$ is even},\\[5pt]
		\big\{ \{0\},\, \{\pm1\},\, \{ \pm 2\}, \ldots,\{\tfrac{n\pm 1}{2}\} \big\} & \text{if $n$ is odd}.
		\end{cases}
	\end{equation*}
	The corresponding supercharacter tables are
	\begin{equation*}\small
		\begin{array}{| c | c c c c c c |}
			\hline
			\Z/n\Z 	& \{0\} 	& \{1,-1\} 	& \{2,-2\} 	& \ldots	& \{\frac{n}{2}-1,\frac{n}{2}+1\} 	& \{\frac{n}{2}\}\\[3pt]
			\{\pm 1\}	&0&1&2&\ldots&\frac{n}{2}-1&\frac{n}{2}\\[3pt]
			\# & 1 & 2 & 2 &\ldots & 2 & 1\\[3pt]
			\hline
			\sigma_1 	& 1 		& 1 	& 1 	& \ldots	& 1 		& 1\\[3pt]
			\sigma_2	& 2	& 2\cos \frac{2\pi}{n} & 2\cos \frac{4\pi}{n} &\ldots & 2\cos \frac{(n-2)\pi}{n} & -2\\[3pt]
			\sigma_3	& 2	& 2\cos \frac{4\pi}{n}& 2\cos \frac{8\pi}{n} &\ldots & 2\cos \frac{2(n-2)\pi}{n} & 2\\[3pt]
			\vdots &\vdots&\vdots&\vdots&\ddots&\vdots&\vdots\\[3pt]
			\sigma_{\frac{n}{2}} & 2& 2\cos \frac{(n-2)\pi}{n}&2\cos \frac{2(n-2)\pi}{n}&\ldots &2\cos \frac{2(\frac{n}{2}-1)^2\pi}{n}&2(-1)^{\frac{n}{2}-1}\\[3pt]
			\sigma_{\frac{n}{2}+1} & 1 & -1 & 1 &\ldots & (-1)^{\frac{n}{2}-1} & (-1)^{\frac{n}{2}} \\[3pt]
			\hline
		\end{array}
	\end{equation*}
	for $n$ even and
	\begin{equation*}\footnotesize
		\begin{array}{| c | c c c c c c |}
			\hline
			\Z/n\Z 	& \{0\} 	& \{1,-1\} 	& \{2,-2\} 	& \ldots	& \{\lfloor \frac{n}{2} \rfloor -1,\lceil \frac{n}{2} \rceil +1\}  	& \{\lfloor \frac{n}{2} \rfloor, \lceil \frac{n}{2} \rceil \}\\[3pt]
			\{1,-1\}	&0&1&2&\ldots&\lfloor \frac{n}{2} \rfloor -1&\lfloor \frac{n}{2} \rfloor\\[3pt]
			\# & 1 & 2 & 2 &\ldots & 2 & 2\\[3pt]
			\hline
			\sigma_1 	& 1 		& 1 	& 1 	& \ldots	& 1 		& 1\\[3pt]
			\sigma_2	& 2	& 2\cos \frac{2\pi}{n} & 2\cos \frac{4\pi}{n} &\ldots & 2\cos \frac{(n-3)\pi}{n} & 2\cos \frac{(n-1)\pi}{n}\\[3pt]
			\sigma_3	& 2	& 2\cos \frac{4\pi}{n}& 2\cos \frac{8\pi}{n} &\ldots & 2\cos \frac{2(n-3)\pi}{n} & 2\cos \frac{2(n-1)\pi}{n}\\[3pt]
			\vdots &\vdots&\vdots&\vdots&\ddots&\vdots&\vdots\\[3pt]
			\sigma_{\lfloor \frac{n}{2} \rfloor } & 2& 2\cos \frac{(n-3)\pi}{n}&2\cos \frac{2(n-3)\pi}{n}&\ldots &2\cos \frac{(n-3)^2\pi}{2n}&2\cos \frac{(n-3)(n-1)\pi}{2n}\\[3pt]
			\sigma_{\lfloor \frac{n}{2} \rfloor+1} & 2 & 2\cos \frac{(n-1)\pi}{n} & 2\cos \frac{2(n-1)\pi}{n} &\ldots & 2\cos \frac{(n-3)(n-1)\pi}{2n} & 2\cos \frac{(n-1)^2\pi}{2n}\\[3pt]
			\hline
		\end{array}
	\end{equation*}
	for $n$ odd.  The corresponding unitary matrix $U$ is a discrete cosine transform (DCT) matrix.
		
\subsection{Gauss sums}\label{SubsectionGauss}
	Let $G = \Z/p\Z$ where $p$ is an odd prime and let $g$ denote
	a primitive root modulo $p$.  We let $\Gamma = \inner{ g^2}$, 
	the set of all nonzero quadratic residues modulo $p$.  
	The action of $\Gamma$ on $G$ results in three superclasses
	$\{0\}$, $\Gamma$, $g\Gamma$,
	with corresponding supercharacter table and symmetric unitary matrix
	\begin{equation*}
		\begin{array}{|c|ccc|}
			\hline
			\Z/p\Z & \{0\} & \Gamma & g\Gamma \\[2pt]
			\inner{g^2} & 1 & \frac{p-1}{2} & \frac{p-1}{2} \\[2pt]
			\hline
			\sigma_1 & 1 & 1 & 1 \\
			\sigma_2 & \frac{p-1}{2} & \eta_0 & \eta_1 \\
			\sigma_3 & \frac{p-1}{2} & \eta_1 & \eta_0 \\
			\hline
		\end{array}
		\qquad\quad
		\underbrace{
		\frac{1}{\sqrt{p}}\!
		\begin{bmatrix}
			1 & \sqrt{ \frac{p-1}{2}} &  \sqrt{ \frac{p-1}{2}} \\
			 \sqrt{ \frac{p-1}{2}} & \eta_0 & \eta_1 \\
			 \sqrt{ \frac{p-1}{2}} & \eta_1 & \eta_0 \\
		\end{bmatrix}
		}_U
	\end{equation*}
	where 
	\begin{equation}\label{eq-Eta}
		\eta_0 = \sum_{h \in \Gamma} e\left( \frac{h}{p} \right),\qquad
		\eta_1 = \sum_{h \in \Gamma} e\left( \frac{gh}{p} \right),
	\end{equation}
	denote the usual quadratic Gaussian periods.  
	
	Clearly the preceding can be generalized to higher-order Gaussian periods in the obvious way \cite{GNGP}.
	If $k|(p-1)$, then we may let $\Gamma = \inner{g^k}$ to obtain the $k+1$ superclasses
	$\{0\},\Gamma, g\Gamma,g^2\Gamma, \ldots, g^{k-1}\Gamma$.  The nontrivial
	superclasses $g^j\Gamma$ each contain $(p-1)/k$ elements, whence \eqref{eq-Uncertainty} yields
	\begin{equation*}
		k+1 = \ceil[\Big]{ \frac{p}{(p-1)/k} }\leq |\supp f| |\supp \widehat{f}|,
	\end{equation*}
	a reasonably strong inequality given that there are only $k+1$ total superclasses.
	
	Let us now return to the quadratic setting $k = 2$ and consider the matrices $T_1,T_2,T_3$
	discussed in Theorem \ref{TheoremPowerful}.  We adopt the labeling scheme $X_1 = \{0\}$,
	$X_2 = \Gamma$, and $X_3 = g\Gamma$.
	Focusing our attention upon $T_2$, we consider
	the constants $c_{2,j,k}$.  A few short computations reveal that the corresponding matrix $[c_{2,j,k}]_{j,k=1}^3$
	of structure constants is given by
	\begin{equation}\label{eq-M}
		\underbrace{
		\megamatrix{0}{1}{0}{ \frac{p-1}{2} }{ \frac{p-5}{4} }{ \frac{p-1}{4} }{0}{ \frac{p-1}{4} }{ \frac{p-1}{4} }
		}_{\text{if $p\equiv 1 \pmod{4}$}}
		\quad\text{or}\quad
		\underbrace{
		\megamatrix{0}{1}{0}{0}{ \frac{p-3}{4} }{ \frac{p+1}{4} }{ \frac{p-1}{2} }{ \frac{p-3}{4} }{ \frac{p-3}{4} }
		}_{\text{if $p\equiv 3 \pmod{4}$}}.
	\end{equation}
	
	For instance, we observe that
	$c_{2,2,2}$ denotes the number of solutions $(x,y)$ in $X_2 \times X_2$ to the equation $x+y = 1$
	(we have selected the representative $z=1$ from the superclass $X_2 = \Gamma$).
	Letting $x = u^2$ and $y = v^2$, the equation $x+y=1$ becomes
	\begin{equation}\label{eq-PythagoreanSquares}
		u^2 + v^2 = 1.
	\end{equation}
	If $t^2 \neq -1$, then one can verify that
	\begin{equation}\label{eq-PythagoreanParameter}
		u = (1 - t^2)(1 + t^2)^{-1}, \qquad v = 2 t( 1 + t^2)^{-1},
	\end{equation}
	is a solution to \eqref{eq-PythagoreanSquares}.  Moreover, every solution $(u,v)$ with $v \neq 0$ to \eqref{eq-PythagoreanSquares} 
	can be parameterized in this manner by setting $t = (1\mp u)v^{-1}$.

	Since $-1$ is a quadratic residue modulo $p$ if and only if $p \equiv 1 \pmod{4}$, we find that \eqref{eq-PythagoreanParameter} 
	produces exactly $p-2$ or $p$ solutions to \eqref{eq-PythagoreanSquares} depending upon whether $p \equiv 1 \pmod{4}$ 
	or $p \equiv 3 \pmod{4}$.   However, we need $x = u^2$ and $y = v^2$ to belong to $X_2=\Gamma$, 
	the set of nonzero quadratic residues in $\Z/p\Z$.  Thus $t = 0, \pm 1$ are ruled out, 
	leaving only $p-5$ (if $p \equiv 1 \pmod{4}$) or $p-3$ (if $p \equiv 3 \pmod{4}$) acceptable values of $t$ 
	that can be used in \eqref{eq-PythagoreanParameter}.  Since there are four choices of sign pairs for 
	$u,v$ leading to the same values of $x,y$, it follows that 
	\begin{equation}\label{eq-m11cases}
		c_{2,2,2}=
		\begin{cases}
			\frac{p-5}{4} &\text{ if $p \equiv 1 \pmod{4}$}, \\[5pt]
			\frac{p-3}{4} &\text{ if $p \equiv 3 \pmod{4}$}.
		\end{cases}
	\end{equation}
	The remaining entries of the matrix \eqref{eq-M} can be computed in a similar manner.  To obtain the matrix $T_2$,
	we weight the numbers $c_{2,j,k}$ appropriately to obtain
	\begin{equation}\label{eq-T2}
		T_2=
		\begin{cases}
			\megamatrix{0}{ \sqrt{ \frac{p-1}{2} } }{0}{ \sqrt{ \frac{p-1}{2} } }{ \frac{p-5}{4} }{ \frac{p-1}{4} }{0}{\frac{p-1}{4}}{\frac{p-1}{4}}
			& \text{if $p \equiv 1\pmod{4}$},\\[25pt]
			\megamatrix{0}{ \sqrt{ \frac{p-1}{2} } }{0}{0}{\frac{p-3}{4}}{\frac{p+1}{4}}{ \sqrt{ \frac{p-1}{2} } }{ \frac{p-3}{4} }{ \frac{p-3}{4} }
			& \text{if $p \equiv 3\pmod{4}$}.
		\end{cases}
	\end{equation}
	Now recall that Theorem \ref{TheoremPowerful} asserts that the eigenvalues of $T_2$ are precisely $\frac{p-1}{2}$, $\eta_0$, and $\eta_1$.
	On the other hand, the eigenvalues of \eqref{eq-T2} can be computed explicitly.  Comparing the two results yields 
	\begin{align*}
		\eta_1 =
		\begin{cases}
			\frac{ -1 \pm \sqrt{p} }{2} & \text{if $p \equiv 1 \pmod{4}$}, \\[5pt]
			\frac{-1 \pm i\sqrt{p} }{2} & \text{if $p \equiv 3 \pmod{4}$},
		\end{cases}
		&&\eta_2 =
		\begin{cases}
			\frac{ -1 \mp \sqrt{p} }{2} & \text{if $p \equiv 1 \pmod{4}$}, \\[5pt]
			\frac{-1 \mp i\sqrt{p} }{2} & \text{if $p \equiv 3 \pmod{4}$}.
		\end{cases}
	\end{align*}
	Among other things, this implies the well-known formula
	\begin{equation*}
		|G_p(a)| = 
		\begin{cases}
			p & \text{if $p|a$}, \\
			\sqrt{p} & \text{if $p \nmid a$},
		\end{cases}
	\end{equation*}
	for the magnitude of the quadratic Gauss sum
	\begin{equation*}
		G_p(a) = \sum_{n=0}^{p-1} \exp\left( \frac{2 \pi i a n^2}{p} \right).
	\end{equation*}

\subsection{Kloosterman sums}\label{SubsectionKloosterman}
	In the following we fix an odd prime $p$.  For each pair $a,b$ in $\Z/p\Z$,
	the \emph{Kloosterman sum} $K(a,b)$ is defined by setting
	\begin{equation*}
		K(a,b) := \sum_{\ell =1}^{p-1} e\left(\frac{a\ell+b\ell^{-1} }{p}\right)
	\end{equation*}
	where $\ell^{-1}$ denotes the inverse of $\ell$ modulo $p$.  
	It is easy to see that Kloosterman sums are always real and that the
	value of $K(a,b)$ depends only on the residue classes of $a$ and $b$ modulo $p$.
	In light of the fact that $K(a,b) = K(1,ab)$ whenever $p \nmid a$, we focus our
	attention mostly on Kloosterman sums of the form $K(1,u)$, adopting
	the shorthand $K_u := K(1,u)$ when space is at a premium.
	Let $G = (\Z/p\Z)^2$ and let 
	\begin{equation*}
		\Gamma = \left\{ \minimatrix{u}{0}{0}{u^{-1}} : u \in (\Z/p\Z)^{\times} \right\}.
	\end{equation*}
	Note that the action of $\Gamma$ on $G$ produces the superclasses
	\begin{equation*}
		\begin{array}{rcl}
			X_1 &=& \big\{ (x,x^{-1}) : x \in (\Z/p\Z)^{\times}\big\},\\[3pt]
			X_2 &=& \big\{ (x,2x^{-1}) : x \in (\Z/p\Z)^{\times}\big\},\\[3pt]
			&\vdots& \\
			X_{p-1} &=& \big\{ (x,(p-1)x^{-1}) : x \in (\Z/p\Z)^{\times}\big\},\\[3pt]
			X_p &=& \big\{(0,1),(0,2),\ldots,(0,p-1)\big\} ,\\[3pt]
			X_{p+1} &=& \big\{(1,0),(2,0),\ldots,(p-1,0)\big\} ,\\[3pt]
			X_{p+2} &=& \big\{(0,0)\big\},
		\end{array}
	\end{equation*}
	and the corresponding supercharacter table
	\begin{equation*}\small
		\begin{array}{|c|cccc|cc|c|}
			\hline
			(\Z/p\Z)^2 & X_1 & X_2 & \cdots & X_{p-1} & X_p & X_{p+1} & X_{p+2} \\[2pt] 
			\Gamma & (1,1) & (1,2) & \cdots & (1,p-1) & (0,1) & (1,0) & (0,0) \\[2pt]
			\# & p-1 & p-1 & \cdots & p-1 & p-1 & p-1 & 1  \\[2pt]
			\hline
			\sigma_1 & K_1 & K_2 & \cdots & K_{p-1} & -1 & -1 & p-1 \\[2pt]
			\sigma_2 & K_2 & K_4 & \cdots & K_{2(p-1)} & -1 & -1 & p-1 \\[2pt]
			\vdots & \vdots & \vdots & \ddots & \vdots & \vdots & \vdots & \vdots \\[2pt]
			\sigma_{p-1} & K_{p-1} & K_{2(p-1)} & \cdots & K_{(p-1)^2} & -1 & -1 & p-1 \\[2pt]
			\hline
			\sigma_p & -1 & -1 & \cdots & -1 & p-1 & -1 & p-1 \\[2pt]
			\sigma_{p+1} & -1 & -1 & \cdots & -1 & -1 & p-1 & p-1 \\[2pt]
			\hline
			\sigma_{p+2} & 1 & 1 & \cdots & 1 & 1 & 1 & 1 \\[2pt]
			\hline
		\end{array}	
	\end{equation*}
	Since $X_i = - X_i$ for all $i$, it follows that the permutation matrix $P$ from Lemma \ref{LemmaU}
	equals the identity.  Among other things, this implies that the unitary matrix 
	\begin{equation}\label{eq-KU}\small
		\underbrace{
		\frac{1}{p}
		\left[
		\begin{array}{cccc|cc|c}
			K_1&K_2&\cdots&K_{p-1}&-1&-1&\sqrt{p-1}\\[2pt]
			K_2&K_4&\cdots&K_{2(p-1)}&-1&-1&\sqrt{p-1}\\[2pt]
			\vdots&\vdots&\ddots&\vdots&\vdots&\vdots&\vdots\\[2pt]
			K_{p-1}&K_{2(p-1)}&\cdots&K_{(p-1)^2}&-1&-1&\sqrt{p-1}\\[2pt]
			\hline
			-1&-1&\cdots&-1& p-1&-1&\sqrt{p-1}\\[2pt]
			-1&-1&\cdots&-1&-1& p-1&\sqrt{p-1}\\[2pt]
			\hline
			\sqrt{p-1}&\sqrt{p-1}&\cdots &\sqrt{p-1}&\sqrt{p-1}&\sqrt{p-1}&1\\[2pt]
		\end{array}
		\right]
		}_U
	\end{equation}		
	is real and symmetric (i.e., $U^2 = I$).
	Moreover, every nontrivial orbit contains exactly $p-1$ elements whence
	\begin{equation*}
		p+2 \leq |\supp f | |\supp \widehat{f} |,
	\end{equation*}
	since $p+1 < p^2/(p-1) <p+2$.  In light of the fact that $|\X|=p+2$, the preceding inequality
	is again quite respectable.

	We remark that the matrix \eqref{eq-KU} is precisely the unitary matrix \cite[eq.~(3.13)]{CKS}, from which dozens of identities 
	for Kloosterman sums may be derived.  The article \cite{CKS} employs the classical character theory of a
	somewhat contrived $4 \times 4$ non-commutative matrix group to obtain the unitarity of this matrix.  We are able to
	accomplish this in less than a page using supercharacter theory.  
	The matrices $T_i$, their remarkable combinatorial properties, and their applications
	are treated in great detail in \cite{CKS}.  We refer the reader there for more information.

\subsection{Heilbronn sums}
	For $p$ an odd prime, the expression
	\begin{equation*}
		H_p(a) = \sum_{\ell=1}^{p-1} e\left( \frac{a\ell^p}{p^2} \right)
	\end{equation*}
	is called a \emph{Heilbronn sum}.  Since $x^p \equiv y^p \pmod{p^2}$ if and only if $x \equiv y \pmod{p}$,
	\begin{equation*}
		\Gamma  = \big\{ 1^p , 2^p,\ldots, (p-1)^p \big\}
	\end{equation*}
	is a subgroup of $(\Z/p^2\Z)^{\times}$ of order $p-1$.  
	Letting $\Gamma$ act upon $G=\Z/p^2\Z$ by multiplication, we obtain the orbits
	\begin{equation*}
		X_1 = g\Gamma,\quad
		X_2 = g^2 \Gamma,\ldots,\quad
		X_{p-1} = g^{p-1} \Gamma,\quad
		X_p = \Gamma,
	\end{equation*}
	\begin{equation*}
		X_{p+1} = \{p,2p,\ldots,(p-1)p\},\quad
		X_{p+2} = \{0\},
	\end{equation*}
	in which $g$ denotes a fixed primitive root modulo $p^2$.	
	For $1 \leq i,j \leq p$, we have
	\begin{equation*}
		\sigma_i(X_j) = \sum_{\ell=1}^{p-1} e\left( \frac{g^{j} (g^{i} \ell^p)}{p^2}\right)
		= \sum_{\ell=1}^{p-1} e\left( \frac{g^{i+j}  \ell^p}{p^2}\right)
		= H_p(g^{i+j}),
	\end{equation*}
	yielding the supercharacter table
	\begin{equation*}\small
		\begin{array}{|c || cc cc|cc |}
			\hline 
			\Z/p^2 \Z& X_1 & X_2 & \cdots & X_p & X_{p+1} & X_{p+2} \\[2pt] 
			\Gamma & g\Gamma & g^2 \Gamma & \cdots & \Gamma & \{p,\ldots,(p-1)p\} & \{0\} \\[2pt]
			\# & p-1 & p-1 & \cdots & p-1 & p-1 & 1 \\[2pt]
			\hline\hline
			\sigma_{1} & H_p(1) & H_p(g) & \cdots & H_p(g^{p-1}) & -1 & p-1\\[2pt]
			\sigma_{2} & H_p(g) & H_p(g^2) & \cdots & H_p(1) & -1 & p-1\\[2pt]
			\vdots & \vdots & \vdots & \iddots & \vdots & -1 & p-1\\[2pt]
			\sigma_{p} & H_p(g^{p-1}) & H_p(1) & \cdots & H_p(g^{p-2}) & -1 & p-1\\ \hline
			\sigma_{p+1} & -1 & -1 & \cdots & -1 & p-1 & p-1\\[2pt]
			\sigma_{p+2} & 1 & 1 & \cdots & 1 & 1  & 1\\ \hline
		\end{array}
	\end{equation*}
	A detailed supercharacter approach to the algebraic properties of Heilbronn sums can be found in \cite{HES}.
	
\subsection{Ramanujan sums}\label{SubsectionRamanujan}
	For integers $n,x$ with $n \geq 1$, the expression
	\begin{equation}\label{eq-RamanujanSumDefinition}
		c_n(x) = \sum_{ \substack{ j = 1 \\ (j,n) = 1} }^n e\left( \frac{jx}{n} \right)
	\end{equation}
	is called a \emph{Ramanujan sum} \cite[Paper 21]{Ramanujan} (see \cite{RSS} for historical references).
	To generate Ramanujan sums as supercharacter values, we first let
	$G = \Z/n\Z$ and $\Gamma = (\Z/n\Z)^{\times}$, observing that
	there exists a $u$ in $\Gamma$ such that $au = b$
	if and only if $(a,n)=(b,n)$.  	
	Let $d_1,d_2,\ldots,d_N$ denote the positive divisors of $n$ and note that 
	the action of $\Gamma$ on $G$ yields the orbits
	\begin{equation*}
		X_i = \{ x : (x,n) = n/d_i \},
	\end{equation*}
	each of size $\phi(d_i)$, and corresponding supercharacters
	\begin{equation}\label{eq-SIDF}
		\sigma_i(\xi) 
		= \sum_{x \in X_i} \psi_x(\xi)
		= \sum_{\substack{ j=1 \\ (j,n)=\frac{n}{d_i}}}^n e\left( \frac{j\xi}{n} \right)
		= \sum_{\substack{k=1\\(k,d_i)=1}}^{d_i} e\left( \frac{k\xi}{d_i} \right) = c_{d_i}(\xi)
	\end{equation}
	(here $\phi$ denotes the Euler totient function).
	The associated supercharacter table is displayed below.
	\begin{equation*}\small
		\begin{array}{|c|cccc|}
			\hline
			\Z/n\Z& X_1 & X_2 & \cdots & X_N \\[2pt]
			(\Z/n\Z)^{\times} & n/d_1 & n/d_2 & \cdots & n/d_N \\[2pt]
			\# & \phi(d_1) & \phi(d_2) & \cdots & \phi(d_N) \\[2pt]
			\hline
			\sigma_1 & c_{d_1}( \frac{n}{d_1}) &  c_{d_1}( \frac{n}{d_2}) & \cdots &  c_{d_1}( \frac{n}{d_N}) \\[2pt]
			\sigma_2 & c_{d_2}( \frac{n}{d_1}) &  c_{d_2}( \frac{n}{d_2}) & \cdots &  c_{d_2}( \frac{n}{d_N}) \\[2pt]
			\vdots & \vdots & \vdots & \ddots & \vdots \\
			\sigma_N & c_{d_N}( \frac{n}{d_1}) &  c_{d_N}( \frac{n}{d_2}) & \cdots &  c_{d_N}( \frac{n}{d_N}) \\[2pt]
			\hline
		\end{array}
	\end{equation*}
	
	Although we have, by and large, avoided focusing on deriving identities and formulas for
	various classes of exponential sums, we can resist the temptation no longer.
	The fact that $c_n(\xi)$ is a superclass function immediately implies that
	\begin{equation}\label{eq-DependsGCD}
		c_n(x) = c_n\big((n,x)\big)
	\end{equation}
	for all $x$ in $\Z$.  In other words, $c_n(x)$ is an \emph{even function modulo $n$} \cite[p.~79]{McCarthyBook},
	\cite[p.~15]{SchwarzBook}.  A well-known theorem from the study of arithmetic functions
	\cite[Thm.~2.9]{McCarthyBook} asserts that	
	if $f:\Z\to\C$ is an even function modulo $n$, then $f$ can be written uniquely in the form
	\begin{equation*}
		f(x) = \sum_{d|n} \alpha(d) c_d(x)
	\end{equation*}
	where the coefficients $\alpha(d)$ are given by
	\begin{equation*}	
		\alpha(d) = \frac{1}{n} \sum_{k|n} f\left( \frac{n}{k}\right) c_{k}\left( \frac{n}{d} \right).
	\end{equation*}
	We now recognize the preceding as being a special case of super-Fourier inversion.
	In contrast, the standard proof requires several pages of elementary but tedious manipulations.

	As another example, we note that the first statement in Lemma \ref{LemmaU} immediately tells us that
	if $d$ and $d'$ are positive divisors of $n$, then
	\begin{equation}\label{eq-RamanujanReciprocity}
		 c_{d}\left( \frac{n}{d'} \right) \phi(d') = c_{d'} \left( \frac{n}{d} \right) \phi(d). 
	\end{equation}
	For our purposes, the importance of \eqref{eq-RamanujanReciprocity}
	lies in the fact that it provides a one-line proof of \emph{von Sterneck's formula}
	(see \cite[Thm.~272]{Hardy}, \cite[Cor.~2.4]{McCarthyBook}, \cite[p.~40]{SchwarzBook})
	\begin{equation}\label{eq-vonSterneck}
		 c_n(x) = \frac{ \mu\left( \frac{n}{(n,x)}\right) \phi(n) }{ \phi\left( \frac{n}{(n,x)}\right)},
	\end{equation}
	in which $\mu$ denotes the M\"obius $\mu$-function.
	Indeed, simply let $d' = n$ and $d = n/(n,x)$ in \eqref{eq-RamanujanReciprocity} and then use \eqref{eq-DependsGCD}
	and the obvious identity $\mu(k) = c_k(1)$.
	We refer the reader to \cite{RSS} for the derivation of even more identities.

	Unlike Gaussian periods and Kloosterman sums, Ramanujan sums are somewhat
	problematic from the perspective of the uncertainty principle.
	Indeed, the denominator of \eqref{eq-Uncertainty} depends upon the size of the largest orbit,
	namely $\phi(n)$, which is often nearly as large as $n$ (e.g., if $p$ is prime, then $\phi(p)=p-1$).  
	This results in a nearly trivial inequality in \eqref{eq-Uncertainty}. 

\subsection{Symmetric supercharacters}\label{SubsectionSymmetric}
	Let $G = (\Z/n\Z)^d$ and let $\Gamma \cong S_d$ be the set of all $d \times d$
	permutation matrices.  Write $d = q n + r$ where $0 \leq r < n$ and consider the vector
	\begin{equation*}
		\vec{x}_0 = (\!\!\!\!\underbrace{1,2,\ldots, n}_{\text{repeated $q$ times}}\!\!\!\!,1,2,\ldots,r),
	\end{equation*}
	for which
	\begin{equation*}
		|\stab(\vec{x}_0)| = \big((q+1)!\big)^r(q!)^{n-r}= (q!)^n(q+1)^r.
	\end{equation*}
	A brief combinatorial argument confirms that $\vec{x}_0$ minimizes $|\stab(\vec{x})|$ whence
	the largest orbit induced by the action of $\Gamma$ on $G$ has order
	\begin{equation*}
		\frac{d!}{(q!)^n(q+1)^r}.
	\end{equation*}
	It now follows from \eqref{eq-Uncertainty} that
	\begin{equation}\label{eq-SymmetricUncertainty}
		 \ceil[\Big]{\frac{n^d (q!)^n(q+1)^r}{d!} } \leq |\supp f| |\supp \widehat{f}|.
	\end{equation}
	Values of these constants for small $n,d$ are given in Table \ref{TableConstants}.
	\begin{table}
		\begin{equation*}\small
			\begin{array}{|c|cccccccccccc|}
				\hline
				d\backslash n & 1 & 2 & 3 & 4 & 5 & 6 & 7 & 8 & 9 & 10 & 11 & 12\\
				\hline
				1& 1 & 2 & 3 & 4 & 5 & 6 & 7 & 8 & 9 & 10 & 11 & 12 \\
				 2&1 & 2 & 5 & 8 & 13 & 18 & 25 & 32 & 41 & 50 & 61 & 72 \\
				 3&1 & 3 & 5 & 11 & 21 & 36 & 58 & 86 & 122 & 167 & 222 & 288 \\
				 4&1 & 3 & 7 & 11 & 27 & 54 & 101 & 171 & 274 & 417 & 611 & 864 \\
				 5&1 & 4 & 9 & 18 & 27 & 65 & 141 & 274 & 493 & 834 & 1,\!343 & 2,\!074 \\
				 6&1 & 4 & 9 & 23 & 44 & 65 & 164 & 365 & 739 & 1,\!389 & 2,\!461 & 4,\!148 \\
				 7&1 & 4 & 11 & 27 & 63 & 112 & 164 & 417 & 950 & 1,\!985 & 3,\!867 & 7,\!110 \\
				 8&1 & 4 & 12 & 27 & 78 & 167 & 286 & 417 & 1,\!068 & 2,\!481 & 5,\!317 & 10,\!665 \\
				 9&1 & 5 & 12 & 35 & 87 & 223 & 445 & 740 & 1,\!068 & 2,\!756 & 6,\!498 & 14,\!219 \\
				 10&1 & 5 & 15 & 42 & 87 & 267 & 623 & 1,\!184 & 1,\!922 & 2,\!756 & 7,\!148 & 17,\!063 \\
				 11&1 & 5 & 16 & 46 & 118 & 291 & 793 & 1,\!722 & 3,\!145 & 5,\!011 & 7,\!148 & 18,\!614 \\
				 12&1 & 5 & 16 & 46 & 147 & 291 & 925 & 2,\!296 & 4,\!717 & 8,\!351 & 13,\!105 & 18,\!614 \\
				 \hline
			\end{array}
		\end{equation*}
		\caption{Values of the expression on the left-hand side of the
		inequality \eqref{eq-SymmetricUncertainty} for the range $1 \leq n,d \leq 12$.}
		\label{TableConstants}
	\end{table}

	Our interest in the exponential sums arising from the action of $S_d$ stems partly from the experimental observation
	that the plots of individual supercharacters $\sigma_X$ are often pleasing to the eye (see Figure \ref{FigureSymmetric}).
	The study of these plots and their properties is undertaken in \cite{FJ2}.
	\begin{figure}[htb]
		\begin{subfigure}{0.3\textwidth}
			\centering
			\includegraphics[width=\textwidth]{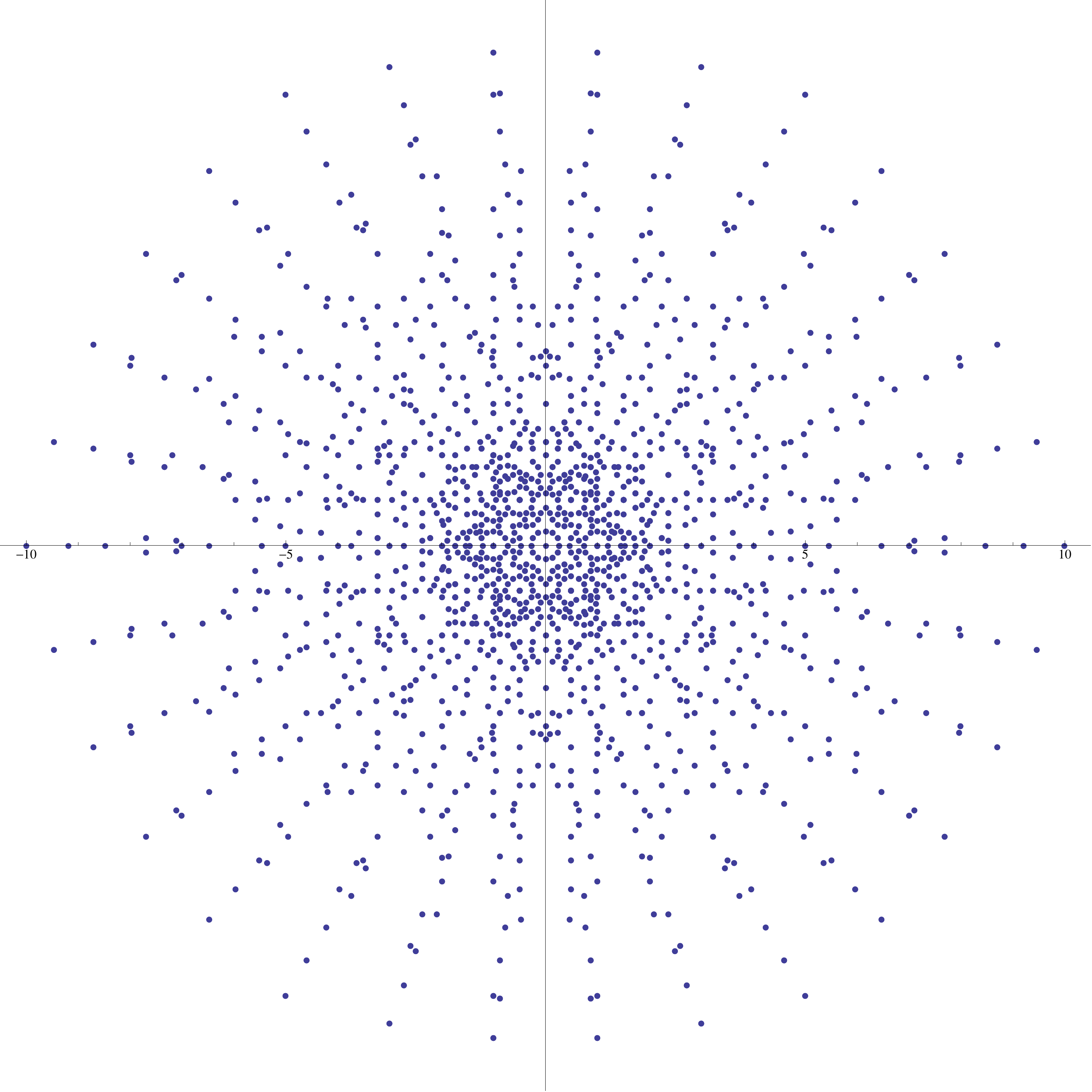}
			\caption{\footnotesize $n=12$, $\vec{x} = (0,0,0,1,1)$}
		\end{subfigure}
		\quad
		\begin{subfigure}{0.3\textwidth}
			\centering
			\includegraphics[width=\textwidth]{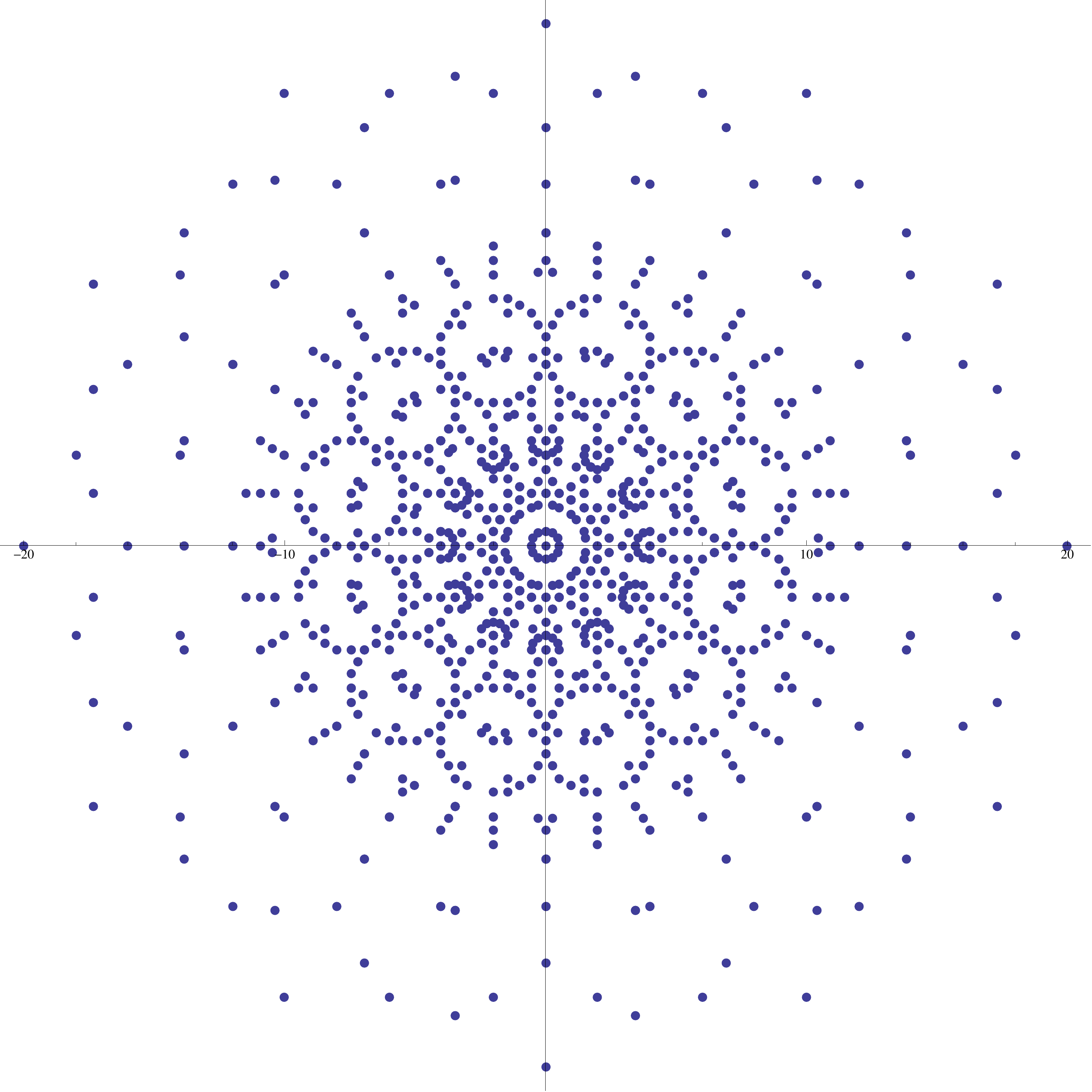}
			\caption{\footnotesize$n=12$, $\vec{x} = (0,0,0,1,6)$}
		\end{subfigure}
		\quad
		\begin{subfigure}{0.3\textwidth}
			\centering
			\includegraphics[width=\textwidth]{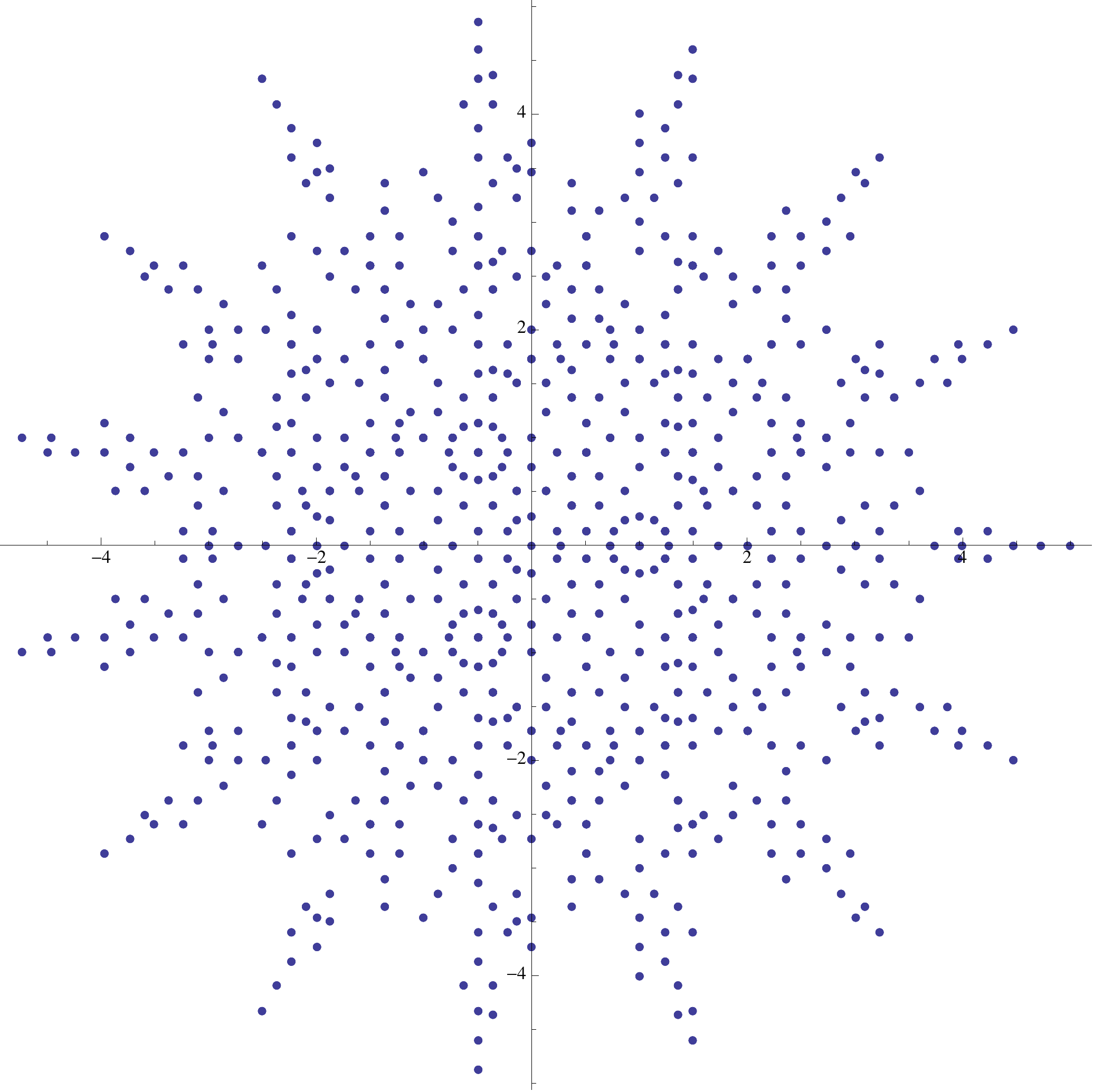}
			\caption{\footnotesize$n\!=\!12$, $\vec{x} = (5,5,5,5,12)$}
		\end{subfigure}	
		\bigskip
		
		\begin{subfigure}{0.3\textwidth}
			\centering
			\includegraphics[width=\textwidth]{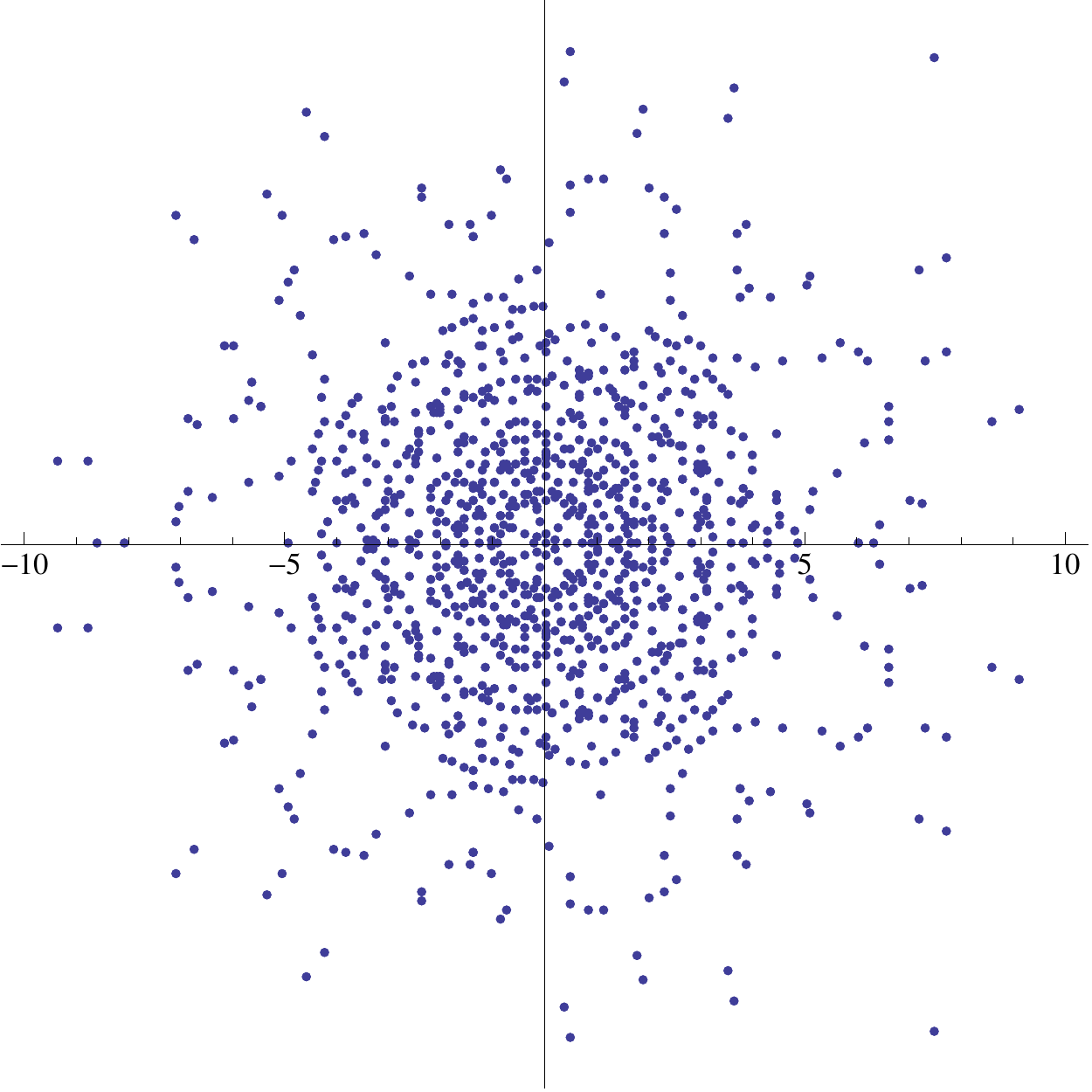}
			\caption{\footnotesize$n=14$, $\vec{x} = (0,1,1,6)$}
		\end{subfigure}
		\quad
		\begin{subfigure}{0.3\textwidth}
			\centering
			\includegraphics[width=\textwidth]{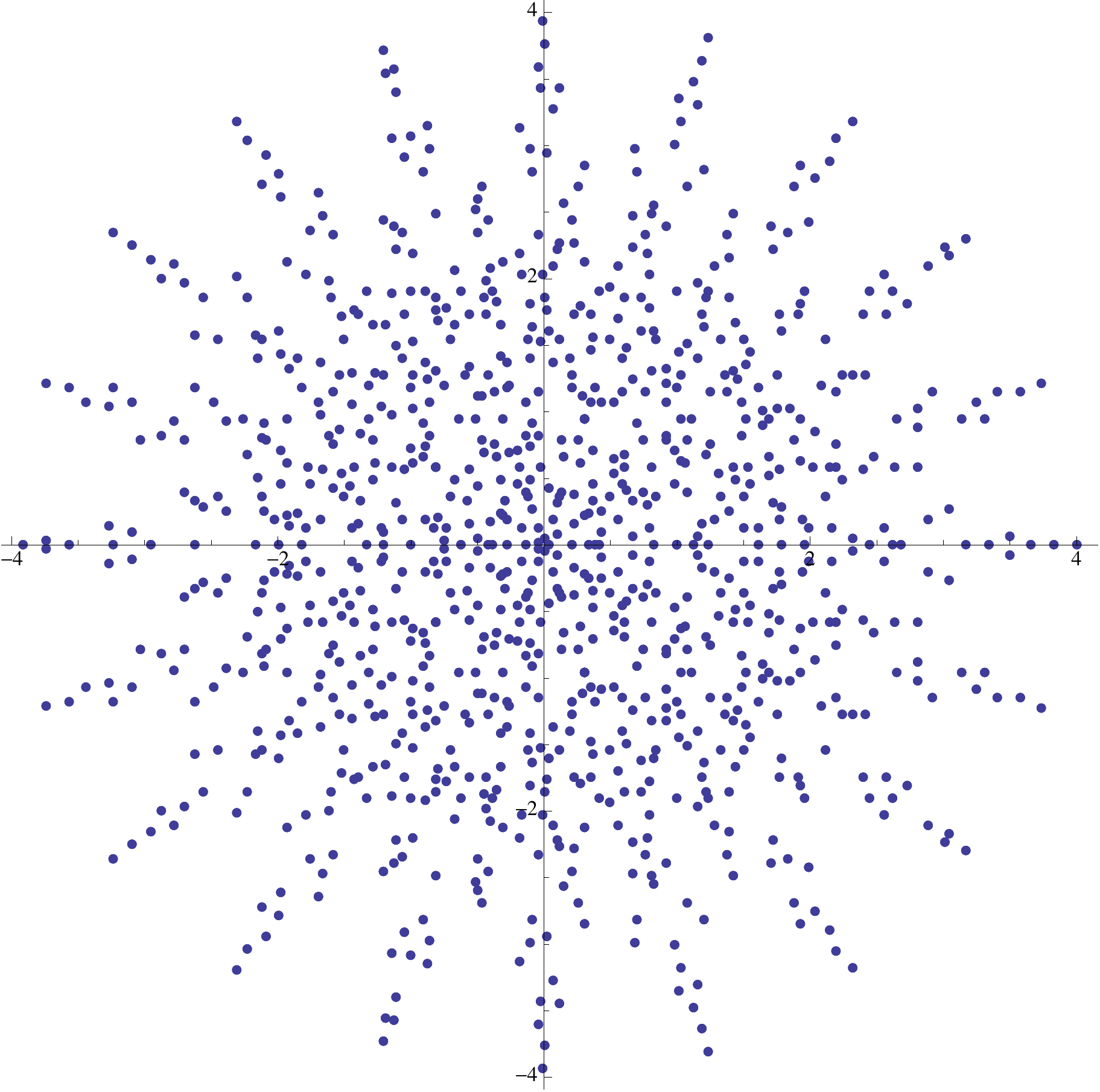}
			\caption{\footnotesize$n=15$, $\vec{x} = (1,1,1,3)$}
		\end{subfigure}
		\quad
		\begin{subfigure}{0.3\textwidth}
			\centering
			\includegraphics[width=\textwidth]{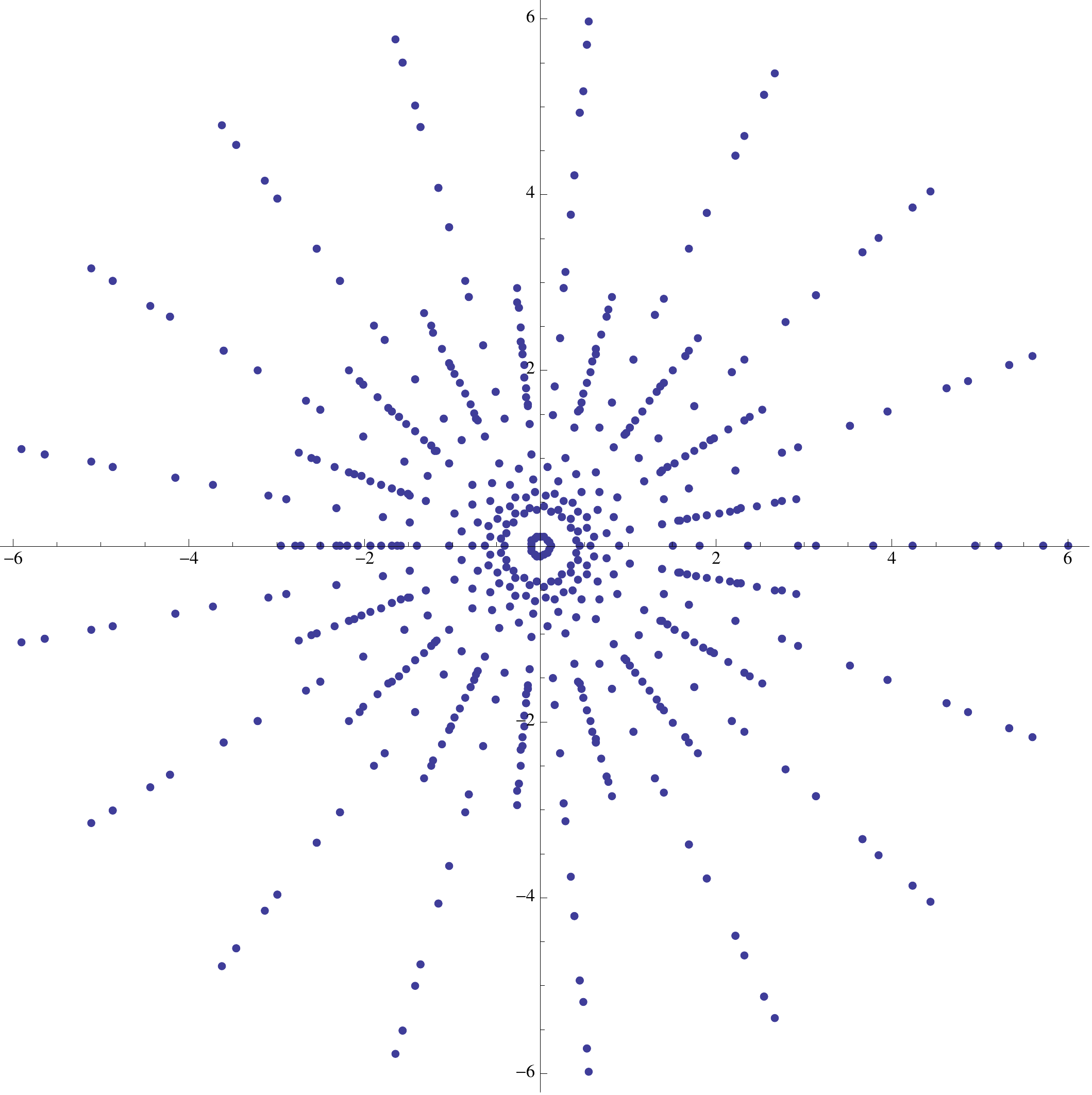}
			\caption{\footnotesize $n=17$, $\vec{x} = (1,2,3)$}
		\end{subfigure}

		\caption{Image of the supercharacter $\sigma_X:(\Z/n\Z)^d\to\C$ where $X = S_d\vec{x}$
		for various $n$, $d$, and $\vec{x}$.}
		\label{FigureSymmetric}
	\end{figure}

\subsection{Upgrading the uncertainty principle?}\label{SubsectionTao}
	Before proceeding, we make a few remarks about T.~Tao's recent strengthening
	of the uncertainty principle for cyclic groups of prime order \cite{Tao} and of the possibility
	of obtaining similar results in the context of super-Fourier transforms.  To be more specific,
	Tao showed that if $p$ is an odd prime, then the classical uncertainty principle for $\Z/p\Z$
	can be improved to
	\begin{equation*}
		p+1 \leq |\supp f| + | \supp \widehat{f}|.
	\end{equation*}

	We argue here, somewhat informally, that such a dramatic improvement cannot be expected
	in the context of supercharacter theories on $(\Z/p\Z)^d$.  Indeed,  Tao's proof relies in a 
	fundamental way on an old result of Chebotar\"ev 
	(see \cite[Lem.~1.3]{Tao} and the references therein), which asserts that every minor of the DFT matrix is invertible.
	This does not, in general, hold for the unitary matrix \eqref{eq-U}, whose adjoint represents the super-Fourier transform
	$\mathscr{F}:\mathcal{S} \to \mathcal{S}$.  
	For instance, the presence of the M\"obius $\mu$-function in von Sterneck's formula 
	\eqref{eq-vonSterneck} indicates that Ramanujan sums frequently vanish.
	Similarly, the unitary matrix obtained in the Kloosterman sum setting has many 
	$2 \times 2$ minors that are singular.

\section{$J$-symmetric groups}\label{SectionJ}

	Throughout the preceding, we have assumed that the group $\Gamma$ which acts on 
	$G= (\Z/n\Z)^d$ is symmetric,
	in the sense that $\Gamma = \Gamma^T$.  However, most of the preceding results also
	hold if $\Gamma$ is merely assumed to be \emph{$J$-symmetric}, meaning that 
	there exists some fixed matrix $J$ in $GL_d(\Z/n\Z)$ such that 
	\begin{equation}\label{eq-J}
		J = J^T,\qquad J\Gamma = \Gamma^T J.
	\end{equation} 
	The reason that we have not pursued this level of generality all along is mostly due to
	the added notational complexity and the fact that plenty of motivating
	examples already exist in the symmetric setting.  
	
	Let us now sketch the modifications necessary to handle the more general setting
	in which $\Gamma$ is $J$-symmetric.
	The first major issue which presents itself is the fact that $\X \neq \Y$.
	As before, the superclasses $Y$ in $\Y$ are orbits $\Gamma \vec{y}$ in $G$ under the action $\vec{y} \mapsto A\vec{y}$
	of $\Gamma$. Identifying the irreducible character $\psi_{\vec{x}}$ with the vector $\vec{x}$ as before,
	the sets $X$ in $\X$ which determine the supercharacters $\sigma_X$
	are orbits under the action $\vec{x} \mapsto A^{-T} \vec{x}$ of $\Gamma$.
	Without the hypothesis that $\Gamma = \Gamma^T$, we cannot conclude that these two actions
	generate the same orbits.  	
	
	Although $\X \neq \Y$ in general, the matrix $J$ furnishes a bijection between $\X$ and $\Y$. 
	Indeed, suppose that $Y = \Gamma \vec{y}$ is the superclass generated by the vector $\vec{y}$
	in $(\Z/n\Z)^d$.  Since $J$ is invertible and $\Gamma$ is a $J$-symmetric group, the set 
	\begin{equation*}
		X = JY  =  J(\Gamma \vec{y}) = \Gamma^{-T} (J\vec{y})
	\end{equation*}
	has the same cardinality as $Y$ and belongs to $\X$.  We therefore enumerate 
	$\X = \{ X_1,X_2,\ldots,X_N\}$ and $\Y = \{Y_1,Y_2,\ldots,Y_N\}$ so that $X_i = JY_i$ 
	and $|X_i| = |Y_i|$ for $i = 1,2,\ldots,N$.  As before, we let $\sigma_i := \sigma_{X_i}$ denote the
	supercharacters associated to the partition $\X$ of $\Irr G$.
	
	In this setting, the unitary matrix \eqref{eq-U} is replaced by the modified matrix
	\begin{equation*}
		U = \frac{1}{\sqrt{n^d}} \left[ \frac{\sigma_i(Y_j) \sqrt{ |Y_j|} }{ \sqrt{ |X_i| } } \right]_{i,j=1}^N,
	\end{equation*}
	whose unitarity can be confirmed using essentially the same computation which we used before.  Showing that $U = U^T$
	requires a little more explanation.  If $Y_i = \Gamma \vec{y}_i$ and $X_j = \Gamma^{-T} \vec{x}_j = \Gamma \vec{x}_j$, then
	\begin{align*}
		| \stab \vec{x}_i |\sigma_i (Y_j)
		&= \sum_{A\in \Gamma} e\left( \frac{A \vec{x}_i \cdot \vec{y}_j }{n} \right) \\
		&= \sum_{A\in \Gamma} e\left( \frac{AJ \vec{y}_i \cdot \vec{y}_j }{n} \right) \\
		&= \sum_{B\in \Gamma} e\left( \frac{JB^T \vec{y}_i \cdot \vec{y}_j }{n} \right) \\
		&= \sum_{B\in \Gamma} e\left( \frac{B^T \vec{y}_i \cdot \vec{x}_j }{n} \right) \\
		&= \sum_{B\in \Gamma} e\left( \frac{B\vec{x}_j \cdot \vec{y}_i }{n} \right) \\
		&= | \stab \vec{x}_j |\sigma_j (Y_i),
	\end{align*}
	where $\vec{y}_j$ denotes the vector $J^{-1} \vec{x}_j$ in $Y_j$.  At this point, the remainder of the proof 
	follows as in the proof of Lemma \ref{LemmaU}.  For each $f:\Y\to\C$, we now define
	\begin{equation*}
		\widehat{f}(X_i) = \frac{1}{\sqrt{n^d}} \sum_{\ell=1}^N f(Y_{\ell}) \overline{(\sigma_{\ell}\circ J)(X_i)},
	\end{equation*}
	so that $\widehat{f}:\X\to\C$.  The corresponding inversion formula is thus
	\begin{equation*}
		f(Y_i) = \frac{1}{\sqrt{n^d}} \sum_{\ell=1}^N \widehat{f}(X_{\ell}) \sigma_{\ell}(Y_i).
	\end{equation*}
	In particular, note that in the $J$-symmetric setting,
	a function and its super-Fourier transform do not share the same domain.
	
\begin{Example}
	Let $p$ be an odd prime, $G = (\Z/p\Z)^2$, and 
	\begin{equation*}
		\Gamma = \left\{ \minimatrix{u}{a}{0}{u} : u \in (\Z/p\Z)^{\times}, a \in \Z/p\Z \right\}.
	\end{equation*}
	Note that $J\Gamma = \Gamma^T J$ where
	\begin{equation*}
		J = \minimatrix{0}{1}{1}{0}.
	\end{equation*}
	The actions $\vec{x} \mapsto A^{-T} \vec{x}$ and $\vec{y} \mapsto A\vec{y}$ of a matrix $A$ in $\Gamma$ yield respective orbits
	\begin{align*}
		X_1 &= \{ (0,0)\},  & Y_1 &= \{ (0,0)\},\\
		X_2 &= \{ (0,u):u\in (\Z/p\Z)^{\times}\}, & Y_2 &= \{ (u,0):u\in (\Z/p\Z)^{\times}\}, \\
		X_3 &= \{ (u,a) : a \in \Z/p\Z, u \in (\Z/p\Z)^{\times}\}, & Y_3 &= \{ (a,u) : a \in \Z/p\Z, u \in (\Z/p\Z)^{\times}\}.
	\end{align*}
	A few simple manipulations now reveal the associated supercharacter table and unitary matrix.
	\begin{equation*}
		\begin{array}{|c|ccc|}
			\hline
			\Z/p\Z & Y_1 & Y_2 & Y_3 \\
			\Gamma & (0,0) & (1,0) & (1,1) \\
			\# & 1 & p-1 & p(p-1)\\
			\hline
			\sigma_1 & 1 & 1 &1 \\
			\sigma_2 & p-1 & p-1 & -1 \\
			\sigma_3 & p(p-1) & -p & 0\\
			\hline
		\end{array}
		\qquad
		\underbrace{
		1/p
		    \begin{bmatrix}
		         1 & \sqrt{p-1} & \sqrt{\left(p-1\right)p} \\
		         \sqrt{p-1} & p-1 & \sqrt{p}\\
		         \sqrt{\left(p-1\right)p}& \sqrt{p} & 0\\
		    \end{bmatrix}
		    }_U
	\end{equation*}
	In particular, observe that $U = U^T$, as expected.
\end{Example}

\bibliography{SESUP}
\bibliographystyle{amsplain}
\end{document}